\documentclass[11pt]{amsart}
\usepackage[utf8]{inputenc}

\usepackage[nospace,noadjust]{cite}
\usepackage{relsize}
\usepackage{cite}
\usepackage{color}
\usepackage[dvipsnames]{xcolor}

\usepackage{tikz-cd}
\usetikzlibrary{calc,fit,matrix,arrows,automata,positioning}
\usetikzlibrary{decorations.pathreplacing,angles,quotes}
\usepackage{colortbl}
\usepackage{hyperref, enumitem}
\hypersetup{
  colorlinks   = true, 
  urlcolor     = blue, 
  linkcolor    = Purple, 
  citecolor   = red 
}
\usepackage{amsmath,amsthm,amssymb}

\usepackage[margin=2.5cm]{geometry}
\usepackage{diagbox}

\usepackage{stmaryrd}

\usepackage{array}
\newcolumntype{x}[1]{>{\centering\arraybackslash\hspace{0pt}}p{#1}}

\theoremstyle{definition}
\newtheorem{theorem}{Theorem}[section]
\newtheorem{definition}[theorem]{{{Definition}}}
\newtheorem{example}[theorem]{{{Example}}}

\newtheorem{remark}[theorem]{{{Remark}}}

\newtheorem{corollary}[theorem]{{{Corollary}}}
\newtheorem{proposition}[theorem]{{{Proposition}}}
\newtheorem{lemma}[theorem]{{{Lemma}}}

\newcommand{\numberset}{\mathbb}
\newcommand{\N}{\numberset{N}}

\newcommand{\Z}{\numberset{Z}}
\newcommand{\Q}{\numberset{Q}}
\newcommand{\R}{\numberset{R}}
\newcommand{\C}{\mathcal{C}}

\newcommand{\F}{\numberset{F}}

\newcommand{\mS}{\mathcal{S}}

\newcommand{\mC}{\mathcal{C}}

\newcommand{\mI}{\mathcal{I}}

\newcommand{\sH}{\sigma}

\newcommand{\mV}{\mathcal{V}}

\newcommand{\D}{{\rm D}}
\newcommand{\dd}{{\rm d}}
\newcommand{\Cl}{{\rm Cl}}

\newcommand{\mH}{\mathcal{H}}

\newcommand{\Fq}{\F_q}

\DeclareMathOperator{\PG}{PG}
\DeclareMathOperator{\PGL}{PGL}

\title[The geometry of intersecting codes]{The geometry of intersecting codes and applications to additive combinatorics and factorization theory}
\usepackage[foot]{amsaddr}
\author{Martino Borello$^{[1,2]}$}
\author{Wolfgang Schmid$^{[1]}$}
\author{Martin Scotti$^{[1]}$}
\address{$^{[1]}$Universit\'e Paris 8, Laboratoire de G\'eom\'etrie, Analyse et Applications, LAGA, Universit\'e Sorbonne Paris Nord, CNRS, UMR 7539, France.}
\address{$^{[2]}$ INRIA Saclay \& LIX, CNRS UMR 7161, École Polytechnique,  France}

\email{martino.borello@univ-paris8.fr, martin.scotti@etud.univ-paris8.fr}
\email{wolfgang.schmid@univ-paris8.fr}

\begin{document}

\begin{abstract}
Intersecting codes are linear codes where every two nonzero codewords have non-trivially intersecting support. In this article we expand on the theory of this family of codes, by showing that nondegenerate intersecting codes correspond to sets of points (with multiplicites) in a projective space that are not contained in two hyperplanes. 
This correspondence allows the use of geometric arguments to demonstrate properties and provide constructions of intersecting codes. We improve on existing bounds on their length and provide explicit constructions of short intersecting codes. Finally, generalizing a link between coding theory and the theory of the Davenport constant (a combinatorial invariant of finite abelian groups), we provide new asymptotic bounds on the weighted $2$-wise Davenport constant. These bounds then yield results on factorizations in rings of algebraic integers and related structures.

\bigskip

\end{abstract}

\maketitle

{
  \hypersetup{linkcolor=black}
  \tableofcontents
}

\section*{Introduction}

Intersecting codes are linear codes for which  every two nonzero codewords have non-trivially intersecting support. 
Intersecting codes are a classical object of study in coding theory introduced in \cite{miklos1984linear,katona1983minimal} and subsequently investigated in many articles (see for example \cite{CL,CZ,retter1989intersecting,cohnen2003intersecting}), but with a primary focus on the binary case. In this case, such codes coincide with minimal codes, which have been intensively studied in the last 20 years. Several practical applications of intersecting and minimal codes are known: they allow communication over AND channels, they may be used in secret sharing schemes, and they are related to other structures such as frameproof codes \cite{blackburn2003frameproof} and $(2,1)$-separating systems \cite{randriambololona20132}. In this article, we primarily focus on the geometric interpretation of intersecting codes, which has remained completely unexplored to date, and to the interactions of these objects with other areas of mathematics, in particular additive combinatorics and algebraic number theory. 

It is well-known that a nondegenerate linear code can be associated with a set of points (with multiplicities) in a projective space and some coding-theoretical properties can be interpreted geometrically. This view is what connects MDS codes to problems with arcs in projective spaces (the famous MDS conjecture was initially formulated as a problem in projective geometry in \cite{segre1955curve}), covering problems to saturating sets, minimal codes to strong blocking sets, etc. Intersecting codes correspond to sets of points that are not contained in any pair of hyperplanes. We will refer to such sets as \emph{non-$2$-cohyperplanar}. This geometric interpretation of intersecting codes allows us to visualize some fundamental properties, but above all, it allows for the introduction of new constructions.

It is clear from the definition of non-2-cohyperplanar sets that adding a point to these sets leaves them non-2-cohyperplanar. Hence, it is fundamental, for constructing purposes, to investigate small sets with this property, possibly minimal with respect to inclusion. We prove some lower bounds on the cardinality of non-2-cohyperplanar sets and a probabilistic existence results. For some low parameters, we provide constructions of the smallest non-2-cohyperplanar sets. We revisit a property proven in \cite{randriambololona20132} regarding intersecting AG codes, along with the concatenation method, to provide explicit constructions of short intersecting codes over any finite field. Quite surprisingly, these explicit constructions improve the probabilistic bound in many cases (more precisely, they almost always improve it for codes over non-prime fields). Moreover, in the binary case it provides the shortest known explicit construction of intersecting codes, which in this case are minimal. Furthermore, we introduce an explicit construction, based on the very recent paper \cite{alon2023strong}, stemming from the union of projective lines, which employs a sufficient geometrical condition called the avoidance property (introduced in \cite{Fancsali}) and some expander graphs. 

The last part of the paper is devoted to the interpretation of our results to additive combinatorics and factorization theory. In particular, continuing the research along the path traced by \cite{SP,multiwise}, we link the theory of intersecting codes to the one of weighted Davenport constants. Actually, the value of this constant is strictly related to 
the function describing the length of the shortest intersecting code for a given dimension and base field, as we will show in Theorem \ref{thm:link}. We will then show the impact of the results on this function and of the explicit constructions of the previous sections on the knowledge of the weighted Davenport constants. Finally, we will explore the connection with factorization theory in algebraic number fields (and more generally certain Dedekind domains and Krull monoids). In particular, at the very end of the paper, we will explore the interplay between problems about intersecting codes over non-prime fields and  Hilbert's ramification theory of some particular number fields with elementary abelian class group. 

\medskip

\noindent \textbf{Outline: } In Section \ref{sec:coding}, we delve into the definition and fundamental properties of intersecting codes. Section \ref{sec:geo} explores the geometry of intersecting codes, particularly showcasing their correspondence with non-$2$-cohyperplanar sets and highlighting their properties and examples. Section \ref{sec:size} addresses the size of small non-$2$-cohyperplanar sets, or equivalently, the length of short intersecting codes, providing both lower and upper bounds. Section \ref{sec:costructions} presents constructions utilizing AG codes and expander graphs. Finally, the paper concludes with applications of the preceding results to the theory of Davenport constants (Section \ref{sec:additive}) and to factorization theory (Section~\ref{sec:factorization}).

\bigskip

\section{Intersecting codes: definitions and fundamental properties}\label{sec:coding}

Throughout the paper $q$ will be a prime power, $\F_q$ will be a finite field of order $q$, and we will endow the vector space $\Fq^{n}$ with the Hamming metric, defined as follows: the \emph{support} of a vector $x \in \Fq^{n}$ is
$$\sigma(x) = \{ i \mid x_{i} \neq 0 \},$$
and its \emph{Hamming weight} is
$${\rm w}(x) = \# \sigma(x).$$
The \emph{Hamming distance} is then defined as ${\rm d}(x, y) = {\rm w}(x-y)$, for $x,y\in \Fq^{n}$.

We start with some classical fundamental definitions.

\begin{definition}
A (linear) \emph{code} $\mathcal{C}$ over a finite field $\Fq$ is a subspace of $\Fq^{n}$. We denote its \emph{dimension} $k = \dim(\mathcal{C})$, and its \emph{minimum distance} $d = {\rm d}(\mathcal{C}) = \min_{c,c' \in \mathcal{C}, c\neq c'} {\rm d}(c,c')$.
We say that a code of dimension $k$ and minimum distance $d$ over $\Fq^{n}$ is an $[n, k, d]_{q}$-code (or an $[n,k]_q$ code if the minimum distance is unknown).\\
A \emph{generator matrix} of an $[n, k, d]_{q}$-code $\mathcal{C}$ is a matrix $G \in \Fq^{k\times n}$ such that $\mathcal{C} = \text{rowsp} (G)$. A code is said \emph{nondegenerate} if no column of $G$ is the zero vector. A code is said \emph{projective} if there are no two linearly dependent columns in $G$. It is straightforward to observe that these last two properties do not depend on the chosen generator matrix. A \emph{parity-check matrix} of an $[n, k, d]_{q}$-code $\mathcal{C}$ is a matrix $H \in \Fq^{(n-k)\times n}$ such that $\mathcal{C}=\{v\in \F_q^n\mid Hv^T={\bf 0}\}$.
\\
Two $[n, k, d]_{q}$-codes $\mathcal{C}$ and $\mathcal{C}'$ are \emph{equivalent} if there is a linear isometry from $\Fq^{n}$ to itself mapping $\mathcal{C}$ to $\mathcal{C}'$ (actually, it is easy to prove that such an isometry should be a monomial transformation).
\end{definition}

The \emph{Singleton bound} is a fundamental concept in coding theory, which provides a relation between the parameters of a code: if $\C$ is an $[n,k,d]_q$ code, then 
$$d\leq n-k+1.$$
If equality holds, the code is called \emph{Maximum Distance Separable} (MDS).

\begin{definition}
A family of codes $\mathcal{F}$ over $\Fq$ is said to be \emph{asymptotically good} if there is an $\varepsilon > 0$ and a sequence of codes $\C_{s} \in \mathcal{F}$ with parameters $[n_{s}, k_{s}, d_{s}]_q$ such that $n_{s} \rightarrow \infty$, as well as both
$$\liminf_{s\rightarrow \infty} \frac{k_{s}}{n_{s}} \geq \varepsilon \quad \text{and} \quad \liminf_{s\rightarrow \infty} \frac{d_{s}}{n_{s}} \geq \varepsilon.$$
\end{definition}

The main object of this paper is the following.

\begin{definition}
A code is called \textit{intersecting} if for any two nonzero codewords the intersection of their supports is nonempty.
\end{definition}

Intersecting codes were first introduced in \cite{miklos1984linear, katona1983minimal} and generalized in \cite{CL} to the case of two distinct codes $\C_{1}$ and $\C_{2}$ for which all codewords $c_{1} \in \C_{1}$ and $c_{2} \in \C_{2}$ share a nonzero coordinate. They have been further investigated in \cite{moreno1992exponential,sloane1993covering,retter1989intersecting,retter1991average}.
In \cite{CZ}, Cohen and Z\'emor provide constructions of asymptotically good intersecting codes and they examine the case where the intersection must have a specific size. All of these contributions focus mainly on the binary case. A study for the case when the base field is of prime cardinality is done (implicitly) in \cite{multiwise}.  As we will see, a particularly relevant construction of explicit families of asymptotically good intersecting codes is provided in \cite{randriambololona20132}.

\medskip

It is useful to define the concatenation of two codes to construct new intersecting codes from old ones.

\begin{definition}
Let $\mI$ be an $[n, k, d]_q$ code and $\mathcal{C}$ be a $[N, K, D]_{q^k}$ code.
Let $\varphi : \Fq^{k} \rightarrow \Fq^{n}$ be a linear map such that ${\rm Im}(\varphi) = \mI$. The \emph{concatenation} of $\mI$ and $\mC$ by $\varphi$ is noted $\mI \square_\varphi \mC$ and is defined as
$$\mI \square_\varphi \mC = \{(\varphi(c_{1}), \dots, \varphi(c_{N})) \mid (c_{1}, \dots, c_{N}) \in \mathcal{C}\}.$$
\end{definition}

It is easy to prove that the code $\mI \square_\varphi \mC$  has parameters $[Nn, Kk, \geq Dd]_q$, for any $\varphi$. In the following, all considered properties will not depend on $\varphi$. Therefore, we will simply denote the concatenation with $\mI \square \mC$.

The following is a straightforward result, which appears also in \cite[Lemma 4.1.]{brassard1996oblivious} and which is used implicitly in \cite{crepeau1993efficient}.

\begin{lemma}\label{lem:concatenation}
Let $\C$ be an intersecting $[N, K, D]_{q^k}$ code and $\mI$ an intersecting $[n, k, d]_q$ code. 
Then $\mI \square \mC$ is an intersecting $[Nn, Kk, \geq Dd]_q$ code.
\end{lemma}

A family directly related to intersecting codes is that of minimal codes. Let us start from their definition.

\begin{definition}
Let $\mC$ be a linear code over $\F_q$.
A nonzero codeword $c \in \mC$ is called \emph{minimal} if for every codeword $c^\prime \in \mC$ such that $\sigma(c^\prime) \subseteq \sigma(c)$, there exists some $\lambda\in\F_q^\ast$ such that $c^\prime=\lambda c$.\\ 
A code is called \emph{minimal} if all its nonzero codewords are minimal.
\end{definition}

Minimal codewords in linear codes were initially investigated in the context of decoding algorithms \cite{MR551274} and have been employed by Massey \cite{Massey} to define the access structure in his code-based secret sharing scheme. The work in \cite{ashikhmin1998minimal} introduced what is now known as the \emph{Ashikhmin-Barg condition}, which serves as a sufficient criterion for code minimality and has been widely utilized in code constructions. In \cite{chabanne2013towards}, minimal codes are investigated in the context of secure two-party computation. Recent research has particularly focused on the parameters of minimal codes, see \cite{ABNgeo, alfarano2022three, chabanne2013towards, cohen2013minimal, lu2021parameters} and short constructions \cite{alon2023strong,bartoli2023small}. As shown in the following straightforward result, minimal codes constitute a subfamily of intersecting codes and they coincide in the binary case.

\begin{lemma}\label{lem:intermin}
Every minimal code is intersecting. Every binary intersecting code is minimal.
\end{lemma}

\begin{proof}
The first part is straightforward: if there are two codewords with non-intersecting support, their sum is a nonzero codeword which is not minimal.\\
Now consider the binary case. If the code is not minimal, then there are two different nonzero codewords $c,c'$ such that $\sigma(c)\subseteq \sigma(c')$. Hence, $\sigma(c+c')\cap \sigma(c')=\emptyset$, so that the code is not intersecting.
\end{proof}

Another family of related codes is that of outer minimal codes, very recently introduced in \cite{ABN2023}.

\begin{definition}
   Let $\mC$ be an $[N,K]_{q^k}$ code. A nonzero codeword $c\in \mC\subseteq  \F_{q^k}^N$ is called ($q$-)\emph{outer minimal}, if, for all $c'\in \mC$, \
     \[\sH(c')\subseteq \sH(c) \wedge \forall i\in \sH(c), \exists \lambda_i\in \F_q \text{ s.t. } c'_i=\lambda_i c_i \ \ \ \Longrightarrow \ \ \ \exists \lambda\in \F_q \text{ s.t. } c'=\lambda c.\]
A code is called ($q$-)\emph{outer minimal} if all its nonzero codewords are ($q$-)outer minimal.
\end{definition}

As shown in \cite{ABN2023}, any outer minimal code concatenated with a minimal code yields a minimal code. This allows to construct short minimal codes and prove some optimal existence results of short minimal codes. Quite unexpectedly, the $2$-outer minimal codes are precisely the intersecting codes.

\begin{proposition}\label{pro:binaryouterminimal}
Let $\mC$ be an $[N,K]_{2^k}$ code. Then $\mC$ is $2$-outer minimal if and only if it is intersecting.
\end{proposition}

\begin{proof}
Suppose that $\mC$ is intersecting. Let $c,c'\in \mC$ nonzero codewords such that $\sH(c')\subseteq \sH(c)$ and $c'_i=c_i$, for all $i\in \sH(c')$. Then $c'= c$, otherwise their sum would be a nonzero codeword with support disjoint from the support of $c'$.

Suppose that $\mC$ is $2$-outer minimal. Let $c,c'\in \mC$ nonzero codewords such that $\sH(c)\cap \sH(c')=\emptyset$. Then $\sH(c)\subseteq \sH(c+c')$ and $c_i=c_i+c_i'$ for all $i\in \sH(c)$ (because $c_i'=0$ for these indexes). However, $c\neq c+c'$, a contradiction.
\end{proof}

Let us conclude the section with a very basic and well-known sufficient condition for a code to be intersecting. Let us remark that this condition can be  seen as an analogue of the previously mentioned Ashikhmin-Barg condition for minimal codes. Indeed, it relies exclusively on the weight distribution of the code.  

\begin{lemma}\label{lem:ABinter}
Let $\C$ be an $[n,k,d]_q$ code. If $2d>n$, then $\C$ is intersecting.
\end{lemma}

\begin{proof}
The statement follows by an elementary pigeonhole argument on the supports of codewords.
\end{proof}

\bigskip

\section{The geometry of intersecting codes}\label{sec:geo}

A classical approach to study linear codes is to consider their geometrical counterparts: researchers have extensively utilized the connections between linear codes and point sets within projective spaces. Notably, the MDS conjecture originated from Segre's inquiry into arcs within finite geometry \cite{segre1955curve}. Other remarkable connections exist between covering codes and saturating sets \cite{davydov2011linear}, or between minimal codes and strong blocking sets \cite{ABNgeo,tang2019full}. The aim of this section is to highlight the geometric interpretation of intersecting codes, which is, to our knowledge, an up to now unexplored topic.\\

Let us start from two basic definitions.

\begin{definition}
In $\Fq^{k}$, define the equivalence relation $\sim$ such that $x \sim y$ if $x$ and $y$ are collinear.
The \emph{projective space} $\PG(k-1, q)$ of dimension $k-1$ over $\F_q$ is then defined as
$$\PG(k-1, q) = \left(\Fq^{k}\setminus\{0\}\right) /{\sim}.$$
\end{definition}

\begin{definition}
A \emph{projective} $[n,k,d]_q$ \emph{system} $\mathcal{P}$ is a finite set of $n$ points (counted with multiplicity) of $\PG(k-1,q)$ that do not all lie on a hyperplane and such that
$$d = n- \max_{\mathcal{H}\text{ hyperplane }}\{|\mathcal{H}\cap \mathcal{P}|\}.$$ 
Projective $[n,k,d]_q$ systems $\mathcal{P}$ and $\mathcal{P}^\prime$ are \emph{equivalent} if there exists some $\varphi \in \PGL(k,q)$ mapping $\mathcal{P}$ to $\mathcal{P}^\prime$ which preserves the multiplicities of the points.
\end{definition}

There is a well-known correspondence between the equivalence classes of nondegenerate $[n,k,d]_q$ linear codes and the equivalence classes of projective $[n,k,d]_q$ systems (see \cite[Theorem~1.1.6]{MR1186841}) which works as follows: let $\mathcal{C}$ be a nondegenerate $[n, k, d]_q$-code, and let $G$ be a generator matrix of $\mathcal{C}$. Since there is no zero column in $G$, it is possible to take the equivalence classes of the columns of $G$ in $\PG(k-1, q)$, say $P_{1}, \dots, P_{n}$. It is straightforward to prove that the multiset $\mathcal{P}_G=\{P_{1}, \dots, P_{n} \}$ is a projective $[n,k,d]_q$ system (see Remark \ref{rmk:pr}). 
Note that if the code $\mathcal{C}$ is projective, all points in $\mathcal{P}_G$ have multiplicity $1$. Varying the generator matrix results in equivalent projective systems. The same holds by taking an equivalent code.

\begin{remark}\label{rmk:pr}
Consider a nondegenerate code $\mC$ with parameters $[n, k, d]_{q}$ and let $G$ be a generator matrix.
Let $\mathcal{P}_G=\{P_{1}, \dots, P_{n} \}$ be defined as above. A codeword $c \in \mC$ is of the form $xG$, where $x \in \Fq^{k}$, and $c_{i} = 0$ if and only if $P_{i}$ belongs to the hyperplane corresponding to $\langle x \rangle^{\perp}$ in $\PG(k-1, q)$.
From our preceding remarks it follows that the support of a codeword corresponds to the points (with multiplicities) in the projective space that are \emph{outside} the corresponding hyperplane.   \end{remark}

Let us introduce a general definition which will be fundamental in describing the geometry of intersecting codes. 

\begin{definition}
Let $t$ be a positive integer. A set of points in a projective space $\PG(k-1, q)$ that is contained in the union of $t$ hyperplanes is called \emph{$t$-cohyperplanar}. A set which is not $t$-cohyperplanar is called \emph{non-$t$-cohyperplanar}.
\end{definition}

\begin{example}
The union of any $t$ hyperplanes in $\PG(k-1, q)$ and one point not in any of these hyperplanes also yields a non-$t$-cohyperplanar set.
\end{example}

Non-2-cohyperplanar sets are the geometric counterpart of intersecting codes, as shown in the following result, which is essentially a remark, but it is named theorem because of its importance for the rest of the paper.

\begin{theorem}\label{thm:geointer}
Let $\mathcal{C}$ be a nondegenerate code over $\Fq$ of dimension $k$ with generator matrix $G$.
Consider the projective system $\mathcal{P}_G\subseteq \PG(k-1, q)$ defined as above. Then $\mathcal{P}_G$ is non-$t$-cohyperplanar if and only if for any set of $t$ codewords the intersection of their support is nonempty. In particular, $\mathcal{C}$ is intersecting if and only if $\mathcal{P}_G$ is non-$2$-cohyperplanar.
\end{theorem}

\begin{proof}
Let $c_{1}=x_1G,\ldots,c_{t}=x_tG$ be two nonzero codewords of $\mC$.
Let $\mH_{1},\ldots,\mH_{t}$ be the projective hyperplanes corresponding to $\langle x_1\rangle^\perp,\ldots,\langle x_t\rangle^\perp$ respectively.
Since $\sigma(c_{j}) = \{ i \mid P_{i} \notin \mH_{j} \}$, the pointset $\mathcal{P}_G$ is not contained in $\bigcup_{j=1}^t \mH_{j}$ if and only if there is an $i \in \bigcap_{j=1}^t\sigma(c_{j})$.
\end{proof}

Since this paper is focused on intersecting codes, from now on we will only consider the case $t=2$. However, in the context of linear frameproof codes \cite{blackburn2003frameproof,huffman2021concise}, considering the general case may be of interest.

\medskip

Let us introduce a related family of geometric structures.

\begin{definition}
A point set $\mS$ in $\PG(k-1,q)$ is called \emph{strong blocking set} if for every hyperplane $\mH\subseteq \PG(k-1,q)$ we have that
$$ \langle \mS \cap \mH\rangle =\mH.$$
\end{definition}

The notion of a \emph{strong blocking set} was first introduced in \cite{davydov2011linear} as a means to construct saturating sets within projective spaces over finite fields. They were later reintroduced in \cite{bonini2021minimal} as \emph{cutting blocking sets}, aiming to generate a family of minimal codes. 

\begin{theorem}[\cite{ABNgeo,tang2019full}]\label{thm:geosbs}
Let $\mathcal{C}$ be a nondegenerate code over $\Fq$ of dimension $k$ with generator matrix $G$.
Consider the projective system $\mathcal{P}_G\subseteq \PG(k-1, q)$ defined as above. Then $\mathcal{P}_G$ is a strong blocking set if and only if $\mathcal{C}$ is a minimal code.
\end{theorem}

As shown in the following, strong blocking sets form a subfamily of non-2-cohyperplanar sets.

\begin{corollary}\label{cor:sbsn2c}
Strong blocking sets in $\PG(k-1,q)$ are non-$2$-cohyperplanar. If $q=2$, the converse is also true.
\end{corollary}

\begin{proof}
This is a direct consequence of Lemma \ref{lem:intermin}, Theorem \ref{thm:geointer} and Theorem \ref{thm:geosbs}. 
\end{proof}

The result above entails that all the numerous known examples of strong blocking sets provide examples of non-2-cohyperplanar sets. We will see, however, that apart from the binary case (where the two concepts coincide), these examples are ``larger'' than necessary (in the sense that they contain many superfluous points).

To illustrate this, let us introduce a general results on non-2-cohyperplanar sets constructed from unions of lines. We first recall the avoidance property, introduced in \cite{Fancsali}.

\begin{definition}
Let $\mathcal{L}$ be a set of lines of $\PG(k-1, q)$.
We say that $\mathcal{L}$ has the \emph{avoidance property} if there is no projective subspace of codimension $2$ that meets every line.
\end{definition}

In \cite{Fancsali}, it is proved that a set of lines having the avoidance property is a strong blocking set (and hence a non-2-cohyperplanar set).
However, as the next result shows, it is enough to take 3 points on each line to get a non-2-cohyperplanar set.

\begin{proposition}\label{pro:3points}
Let $\mathcal{L}$ be a set of lines with avoidance property and let $\mathcal{S}$ be a set of points such that for every $\ell\in \mathcal{L}$, 
$$|\mathcal{S}\cap \ell|\geq 3.$$ 
Then $\mathcal{S}$ is a non-2-cohyperplanar set.
\end{proposition}

\begin{proof}
Let $\mathcal{H}_{1}$ and $\mathcal{H}_{2}$ be two projective hyperplanes and let $\mathcal{V} = \mathcal{H}_{1} \cap \mathcal{H}_{2}$.
The subspace $\mathcal{V}$ is of codimension $2$, so there exists $\ell \in \mathcal{L}$ such that $\ell \cap \mathcal{V} = \emptyset$. The line
$\ell$ cannot be contained in $\mathcal{H}_{1}$, since otherwise $\ell$ would meet also $\mathcal{H}_{2}$, contradicting $\ell \cap \mathcal{V} = \emptyset$. 
Symmetrically, $\ell$ also cannot be contained in $\mathcal{H}_{2}$. 
This means that 
$$|\ell \cap \mathcal{H}_{1}|=|\ell \cap \mathcal{H}_{2}|=1.$$ 
Therefore, since $\mathcal{S}$ contains at least $3$ points of $\ell$, $\mathcal{S}$ contains at least one point outside of $\mathcal{H}_{1} \cup \mathcal{H}_{2}$, meaning that $\mathcal{S}$ is non-2-cohyperplanar.
\end{proof}

\begin{remark}
Proposition \ref{pro:binaryouterminimal} and Theorem \ref{thm:geointer} imply that, when $q$ is even, the geometric counterparts of $2$-outer minimal codes, that is the $2$-outer strong blocking sets, are non-2-cohyperplanar. See \cite{ABN2023} for more details about $2$-outer strong blocking sets and their geometric properties. 
\end{remark}

We will now give two examples of non-2-cohyperplanar sets of quite small size (with respect to the projective space in which they are defined).

\begin{example}[\textbf{Arcs with at least $2k-1$ points}]\label{exa:arcs}
An \emph{arc} in $\PG(k-1,q)$ is a set of points with the property that any $k$ of them span the whole space. It is well-known that arcs in $\PG(k-1,q)$ correpond to MDS codes of dimension $k$ over $\F_q$.
Any arc $\mathcal{A}$ with at least $2k-1$ points is non-$2$-cohyperplanar: the maximum number of points of $\mathcal{A}$ contained in a hyperplane is $k-1$, by definition. Hence, if $|A|>2(k-1)$, for any couple of hyperplanes $\mathcal{H}_1,\mathcal{H}_2$ there is always a point of $\mathcal{A}$ not contained in $\mathcal{H}_1\cup \mathcal{H}_2$.
\end{example}

\begin{example}[\textbf{The Sparse Tetrahedron}]\label{exa:tetra}    
Consider $k$ points $V_1,\ldots,V_k$ of $\PG(k-1, q)$ spanning the whole space. For any $i,j\in\{1,\ldots,k\},i<j$, consider a point $P_{i,j}$ on the line $\langle V_i,V_j\rangle$, $P_{i,j}\not\in\{V_i,V_j\}$. The set
\[\mathcal{T}=\{V_1,\ldots,V_k\} \cup \{P_{i,j}\mid i,j\in\{1,\ldots,k\},i<j\}\]
is called \emph{sparse tetrahedron}.
Such a 
set is non-2-cohyperplanar. Indeed, the intersection of $\mathcal{T}$ with any hyperplane $\mH$ cannot contain all $V_1,\ldots,V_k$. 
If $\mathcal{T}\cap \mH$ does not contain $V_i$, it also does not contain any line passing through $V_i$, and in particular, it does not contain any of the lines $\langle V_i,V_j\rangle$.
For each of these lines, there is at least a point distinct from $V_i$ not contained in $\mH$ (otherwise the whole line would be contained in $\mH$). Therefore, we have identified a set of points not contained in $\mH$ spanning the whole space.
This means that any other hyperplane $\mH'$ will not suffice to cover all the points, guaranteeing that we have a non-2-cohyperplanar set.
\end{example}

We may introduce a notion of minimality.

\begin{definition}
A non-2-cohyperlanar set $\mathcal{S}$ is said to be \emph{minimal} if there exist two hyperplanes $\mathcal{H}_1$ and $\mathcal{H}_2$ and a point $P\in\mathcal{S}$ such that $\mathcal{S}\setminus\{P\}\subset \mathcal{H}_1\cup\mathcal{H}_2$.  
\end{definition}

\begin{remark}
It is easy to observe that arcs with $2k-1$ points in $\PG(k-1,q)$ and sparse tetrahedrons are examples of minimal non-2-cohyperplanar sets.
\end{remark}

Let us conclude the section with a geometric property of non-2-cohyperplanar sets, with a strong coding theoretical implication.

\begin{theorem}\label{thm:distance}
Let $\mS$ be a non-2-cohyperplanar set and $\mH$ be a hyperplane in $\PG(k-1, q)$. Then 
$$|\mS\setminus (\mS\cap \mH)|\geq k.$$
Equivalently, if $\mathcal{C}$ is an $[n,k,d]_q$ intersecting code, then $d\geq k$.
\end{theorem}

\begin{proof}
The set $\mS\setminus (\mS\cap \mH)$ cannot be contained in a hyperplane, since otherwise $\mS$ would be $2$-cohyperplanar. Hence, it must contains at least $k$ points. The statement about intersecting codes is a direct consequence of Theorem \ref{thm:geointer}.
\end{proof}

\bigskip

\section{On the size of small non-2-cohyperplanar sets}\label{sec:size}

While it is easy to give examples of non-2-cohyperplanar sets with many points, it is not clear how small they can be. Moreover, if we add a point to a set that is already non-2-cohyperplanar, it remains non-2-cohyperplanar. Hence, finding the minimum size of a non-2-cohyperplanar set in $\PG(k-1, q)$ is a natural problem and this is the aim of this section.

\begin{definition}
We define $i(k, q)$ to be the \emph{size of the smallest non-$2$-cohyperplanar set} in $\PG(k-1, q)$ (equivalently, the \emph{length of the shortest linear intersecting code} over $\Fq$ of dimension $k$).
\end{definition}

\subsection{Lower bounds}

We start with an easy lower bound.

\begin{theorem}\label{thm:bound}
Let $\mS$ be a non-2-cohyperplanar set in $\PG(k-1,q)$. Then $|\mS|\geq 2k-1$. Hence
$$i(k,q)\geq 2k-1.$$
If $|\mS|=2k-1$, then $\mS$ is an arc.
\end{theorem}

\begin{proof}
Suppose $|\mS| \leq  2k-2$. 
Since any $k-1$ points are always contained in a hyperplane, there must exist two hyperplanes containing $\mS$, which gives a contradiction, proving the inequality.\\
Suppose now that $|\mS|=2k-1$. In coding theoretical language, this means that the corresponding code has length $n=2k-1$. Now, its minimum distance $d$ satisfies $d\geq k$, by Theorem \ref{thm:distance}, and 
$$d\leq 2k-1-k+1=k,$$
by the Singleton bound. Hence $d=k$ and the code is MDS, so that $\mS$ is an arc.
\end{proof}

\begin{remark}\label{rmk:MDS}
Let us underline that Theorem~\ref{thm:bound} implies that the sets introduced in Example~\ref{exa:arcs} are the smallest whenever arcs of cardinality $2k-1$ exist. It is well-known that arcs with $q+1$ points exist. The celebrated MDS conjecture, posed by Segre in \cite{segre1955curve}, states that the maximal size of an arc in $\PG(k-1,q)$ (where $2\leq k\leq q-1$) is $q+1$, up to two exceptional cases for which it is $q+2$. Moreover, if $k\geq q$, the maximal size of an arc is $k+1$. The conjecture has been demonstrated across various parameter sets $q$ and $k$ (see \cite{hirschfeld2015open} for a survey). In \cite{ball2012sets}, Ball achieved a significant breakthrough by demonstrating that the MDS conjecture holds when $q$ is a prime. In particular, if $q$ is a prime, he proved that every arc with $q+1$ points in $\PG(k-1,q)$, with  $2\leq k\leq q-1$, is a rational normal curve (the geometric counterpart of Reed-Solomon codes). So, under the MDS conjecture,
$$i(k,q)=2k-1$$
if and only if $k\leq \frac{q+2}{2}$ (or $k\leq \frac{q+3}{2}$ in the exceptional cases). Note that it is not necessary to invoke the MDS conjecture to say that the bound is not tight for large $k$, because it is well-known that, if $2\leq k\leq q-1$, arcs cannot exist for $k$ larger than $2q-2$ (see \cite[Corollary 7.4.4]{huffman2010fundamentals}).
\end{remark}

The following result is a Plotkin-like bound.

\begin{theorem}\label{thm:Plotkin}
For $1 \leq t \leq k$,
\begin{equation}\label{eq:Plotkin}
i(k,q) \geq k + \frac{q^{t}-1}{q^{t} - q^{t-1}}(k-t).
\end{equation}
\end{theorem}

\begin{proof}
Let us give a proof in coding theoretical language. Let $G = (I_{k} \mid A)$ be a generator matrix of an intersecting $[n,k,d]_q$ code $\C$, and consider $t$ rows of $G$.
Let $\mV$ be the multiset of vectors formed by the last $n-k$ coordinates of the nonzero vectors in the rowspan of these $t$ rows.
We aim to compute the sum of the weights of all vectors in $\mV$, called the total weight of $\mV$ and denoted ${\rm w}(\mV)$.
Since the $t$ rows are linearly independent, there are $q^{t} - 1$ vectors in $\mV$ (counted with multiplicities).
Since any codeword corresponding to a vector of $\mV$ has at most $t$ nonzero elements in the first $k$ coordinates, each element of $\mV$ must have weight at least $d - t \geq k - t$ (the last inequality comes from Theorem \ref{thm:distance}).
This means that 
$${\rm w}(\mV) \geq (q^{t} - 1) (k - t).$$
On the other hand, each vector of $\mV$ has length exactly $n - k$, and for each coordinate there are at most $q^{t} - q^{t-1}$ codewords of $\mV$ that are nonzero at this coordinate.
This yields that the total weight is at most
$${\rm w}(\mV) \leq (n - k) (q^{t} - q^{t-1}).$$
Combining both inequalities finishes the proof.
\end{proof}

\begin{remark}
Note that \eqref{eq:Plotkin} reduces to the bound in Theorem \ref{thm:bound} for $t=1$. By Remark \ref{rmk:MDS}, the case $t=1$ give the best bound for small $k$. When $q$ is large, the case $t=1$ also gives the best bound. So Theorem \ref{thm:Plotkin} improves on Theorem \ref{thm:bound} only when $k$ is large compared with $q$.
\end{remark}

\begin{remark}
For $q=2$, non-2-cohyperplanar sets coincide with strong blocking sets. In \cite{alfarano2022three}, it is proved that $i(k,2)\geq 3k-3$. This means that \eqref{eq:Plotkin} is never tight for $q=2$, $k>2$ and any $t\in \{1,\ldots,k\}$. Note also that some structural results on non-2-cohyperplanar sets of cardinality $3k-3$ are given in \cite{scotti2024lower}.
\end{remark}

\subsection{Asymptotic lower bounds}

After demonstrating general bounds, we now present asymptotic bounds, which improve upon previous ones for large $k$ (with fixed $q$). It is worth noting that similar methods have also been employed for general strong blocking sets in \cite{scotti2024lower,bishnoi2023blocking} and for intersecting codes over prime fields in \cite{katona1983minimal,CL,multiwise}.

Before stating our bound, we must define $q$-ary upper-bounding functions.

\begin{definition}
Let $\mC$ be a code with parameters $[n, k, d]_{q}$.
We define its \emph{rate} $R=k/n$ and its \emph{relative minimum distance} $\delta = d/n$.
Let $f: [0, 1] \rightarrow [0, 1]$ be a continuous decreasing function. We say that $f$ is \emph{$q$-ary upper-bounding} if, for any $R$ and $\delta$ verifying $R > f(\delta)$, there is no sequence of codes with parameters $[n_{s}, R\cdot n_{s}, \delta \cdot n_{s}]_{q}$ such that $n_{s} \rightarrow \infty$.
\end{definition}

\begin{example}
The Singleton bound stated in Section \ref{sec:coding} implies that the function $f(\delta) = 1-\delta$ is $q$-ary upper-bounding for any $q$.
\end{example}

Notice that, by Theorem \ref{thm:distance}, the parameters of intersecting codes must lie in the region $$\{(\delta, R)\in \R_{\geq 0}^2 \mid R\leq \delta \}.$$
By combining this with a suitable $q$-ary upper-bounding function, one may deduce an upper bound on the asymptotic rate of intersecting codes over $\Fq$, or, equivalently, a lower bound on their length.

\begin{theorem}\label{thm:lowerasymp}
$$\liminf_{k \rightarrow \infty} \frac{i(k, q)}{k} \geq 2 + \frac{1}{q-1}.$$
\end{theorem}

\begin{proof}
The asymptotic Plotkin bound \cite[Theorem 2.10.2]{huffman2010fundamentals} corresponds to the $q$-ary upper-bounding function
$$f(x) = 1 - \frac{q}{q-1} \ x.$$
The intersection of the graphs of this function and the function $g(x) = x$ corresponding to the bound $R\leq \delta $ is the point $\left(\frac{q-1}{2q-1}, \frac{q-1}{2q-1}\right)$.
This produces the upper bound on the asymptotic rate and hence the lower bound above.
\end{proof}

\begin{remark}
Note that the previous result can be also obtained by Theorem \ref{thm:Plotkin}, by letting $t$ grow to infinity.  
\end{remark}

By considering the MRRW bound \cite[Theorem 2.10.6]{huffman2010fundamentals}, instead of the Plotkin bound, one is able to provide a stronger upper bound on the rate for small values of $q$.

Let $H_{q}$ be the $q$-ary entropy function, that is
$$H_{q}(x) = -x\log_{q}\Big(\frac{x}{q-1}\Big) - (1-x)\log_{q}(1-x).$$
The MRRW bound states that 
$$M_{q}(x) = H_{q}\left(\frac{1}{q}\left(q-1-(q-2)x - 2 \sqrt{(q-1)x(1-x)}\right)\right)$$
is a $q$-ary upper-bounding function. Using the same argument as above, we get stronger upper bounds on the maximum rate of intersecting codes whenever $q\leq 17$.

Table \ref{table1} summarizes the improved bounds obtained this way (some of them were already known in the references cited above). 

\begin{center}
\begin{table}[ht!]
\caption{Lower bound on the asymptotic length of intersecting codes}
\label{table1}
\begin{tabular}{c|c}
$q$ & $\liminf_{k \rightarrow \infty} \frac{i(k, q)}{k}$ \\ \hline
2   & 3.5276                                            \\
3   & 2.8272                                            \\
4   & 2.5713                                            \\
5   & 2.4342                                            \\
7   & 2.2862                                            \\
8   & 2.2411                                            \\
9   & 2.2060                                            \\
11  & 2.1547                                            \\
13  & 2.1185                                            \\
16  & 2.0802                                            \\
17  & 2.0703                                            \\
\end{tabular}
\end{table}
\end{center}

\begin{theorem}
The MRRW bound yields a better upper bound on the rate of intersecting codes than the Plotkin bound when $q \leq 17$.
In other words, if and only if $q \geq 19$, the following holds :
$$M_{q}\left(\frac{q-1}{2q-1}\right) \geq \frac{q-1}{2q-1}.$$
\end{theorem}

\begin{proof}
First we must compute $A(x) = \frac{1}{q}\left(q-1-(q-2)x - 2 \sqrt{(q-1)x(1-x)}\right)$ with $x = \frac{q-1}{2q-1}$.
This yields
$$A\left(\frac{q-1}{2q-1}\right) = \frac{(q-1)(\sqrt{q}-1)^{2}}{q(2q-1)}.$$

We now want to know for what values of $q$ it is true that
$$M_{q}\left(\frac{q-1}{2q-1}\right) \geq \frac{q-1}{2q-1}.$$

Now, for simplicity, set $B(q) = q(2q-1) - (q-1)(\sqrt{q} - 1)^{2}$, $C(q) = (2q-1)q$ and $D(q) = (q-1)(\sqrt{q}-1)^{2}$, and define $g(x) = x\log_{q}(x)$. Note that $B(q)$, $C(q)$ and $D(q)$ are all positive.
By straightforward computations, the above inequality is equivalent to
$$g(B(q)) - (q-1)g\left(\frac{D(q)}{q-1}\right) + g(C(q)) \geq q(q-1).$$

Since $g$ is a convex function, and since $D(q) = C(q) - B(q)$, we can give the following lower bound
$$g(C(q))-g(B(q)) \geq D(q)g'(B(q)) = D(q) \log_{q}(eB(q)).$$

Therefore it is sufficient to establish
$$D(q) \log_{q}(eB(q)) - D(q)\log_{q}((\sqrt{q} - 1)^{2}) \geq q(q-1)$$
and this simplifies to
$$(\sqrt{q} -1)^{2} \log_{q}\left(\frac{eB(q)}{(\sqrt{q} - 1)^{2}}\right) \geq q.$$

Since $(\sqrt{q} - 1)^{2} \leq q$, in order for the above inequality to be true it is enough to have
\begin{align*}
(\sqrt{q} -1)^{2} \log_{q}\left(\frac{eB(q)}{q}\right) &\geq q \\
(\sqrt{q} -1)^{2} \log_{q}(eq) &\geq q \\
q &\geq 2\sqrt{q}(\ln(q) + 1) \\
1 + \frac{1}{2} \cdot \sqrt{q} &\geq \ln(q)
\end{align*}

Writing $f(x) = 1 + \frac{1}{2} \cdot \sqrt{x} - \ln(x)$, it is easy to check that $f$ is increasing as soon as $x \geq 16$.

In particular, straightforward computation shows that $f(144) > 0$, meaning that as soon as $q \geq 144$ we have
$$M_{q}\left(\frac{q-1}{2q-1}\right) \geq \frac{q-1}{2q-1}.$$

For the remaining values of $q$, that is for $19 \leq q \leq 144$, the theorem can be checked by direct computation.
\end{proof}

\subsection{Upper bounds}

The sparse tetrahedron construction, presented in the Example \ref{exa:tetra}, yields a non-2-cohyperplanar set of size $k(k+1)/2$ in $\PG(k-1, q)$, therefore providing an upper bound on $i(k, q)$, namely
$$i(k, q) \leq \frac{k(k+1)}{2}.$$
Note that this construction correspond to codes with parameters $\left[\frac{k(k+1)}{2}, k, k\right]$ (the minimum distance can be easily obtained by a geometric argument). In particular, this is not a family of asymptotically good codes.

The following result is an upper bound obtained from a probabilistic existence result (equivalent to taking random points in the projective space). Let us highlight the fact that this is already known for prime fields: for $q = 2$ it is due to Koml\'os (unpublished proof, 1983, cited in \cite{CL}), and for the more general case when $q$ is a prime, we know of no earlier proof than \cite[Theorem 7.3]{multiwise}. Our proof follows the same arguments and we write it explicitly for the sake of completeness.

Let us recall that the \emph{Gaussian coefficient} ${N \brack K}_{q}$ is the number of subspace of dimension $K$ is a vector space of dimension $N$ over $\F_q$, that is 
$${N \brack K}_{q}=\prod_{i=0}^{K-1}\frac{q^N-q^i}{q^K-q^i}.$$

\begin{theorem}\label{prob}
If 
$$n\geq \frac{2}{\log_{q}(\frac{q^{2}}{2q-1})}\ k$$
then an $[n,k,d]_q$ intersecting code, or equivalently, a non-2-cohyperplanar set in $\PG(k-1,q)$ of cardinality $n$, exists. Hence
$$\limsup_{k \rightarrow \infty} \frac{i(k, q)}{k} \leq \frac{2}{\log_{q}(\frac{q^{2}}{2q-1})}.$$
\end{theorem}

\begin{proof}
We follow the classical counting arguments used, for example, in the well-know Gilbert–Var\-sha\-mov bound (see \cite[Theorem 2.10.8]{huffman2010fundamentals}). Again, we will use coding-theoretical language. 

Let 
$$\mathcal{B}_n=\{\{x,y\}\subseteq \F_q^n\mid \sigma(x)\cap\sigma(y)=\emptyset\}.$$ 
For each coordinate $i\in\{1,\ldots,n\}$ of a pair of vectors $\{x,y\}$ in $\mathcal{B}$ we have three possibilities: 
$$(x_i\neq 0\wedge y_i=0) \vee (x_i=0\wedge y_i\neq 0) \vee (x_i=0 \wedge y_i=0).$$
Hence $|\mathcal{B}|=(2(q-1)+1)^n=(2q-1)^n$.

Let 
$$\mathcal{F}_{n,k}=\{\C\subseteq \F_q^n\mid \dim\C =k \},$$
whose cardinality is clearly ${n \brack k}_{q}$.

Now, each pairs of vectors in $\mathcal{B}_n$ is contained in exactly ${n-2 \brack k-2}_{q}$ elements of $\mathcal{F}_{n,k}$. A code in $\mathcal{F}_{n,k}$ is interesting if and only if it does not contain any element of $\mathcal{B}_n$. Since there are at most 
$$(2q-1)^{n} \cdot{n-2 \brack k-2}_{q}$$
codes in $\mathcal{F}_{n,k}$ which contains an element of $\mathcal{B}_n$, if 
$$(2q-1)^{n} \cdot{n-2 \brack k-2}_{q} \leq {n \brack k}_{q}$$
then there exist intersecting codes with parameters $[n, k]_{q}$. By straightforward calculations we get that the above condition is implied by $q^{n\log_{q}(2q-1) + 2(k-n)} \leq 1$, hence the statement.
\end{proof}

\begin{corollary}
   Intersecting codes are asymptotically good.
\end{corollary}

\begin{proof}
    Theorem \ref{prob} and Theorem \ref{thm:distance} yield that a family of 
$$\left[\frac{2}{\log_{q}(\frac{q^{2}}{2q-1})}\ k,k,\geq k\right]_q$$
intersecting codes exist. This is an asymptotically good family.
\end{proof}

\begin{remark}
Even though Theorem \ref{prob} provides a very good upper bound (converging to the same value as the lower bound for large $q$), it has the drawback of not providing guidance on how to \emph{explicitly} construct such small cardinality sets. The following section will be dedicated to such constructions.
\end{remark}

\subsection{Explicit examples for low dimensions and small base fields}

We end this section with a list of explicit computations of the actual value of $i(k, q)$ for small values of $k$ and $q$. We summarize in Table \ref{table2} all the known results, whose proof is given below. Whenever we write $[n_1,n_2]$ we mean that $i(k,q)$ is not known but contained in this interval. The colors indicate the argument used to prove the lower or upper bound, as we will explain at the end of this subsection.

\begin{table}[ht!]
\centering
\caption{Values of $i(k,q)$ for small $q$ and $k$}
\label{table2}
\begin{tabular}{c||c|c|c|c|c|c|c|c|c|}
\diagbox[width=1cm]{$q$}{$k$} & 2 & 3 & 4 & 5 & 6 & 7 & 8 & 9 \\
\hline
\hline
2 & 3 & 6 & 9 & 13 & 15 & 20 & 24 & 26 \\
\hline
3 & 3 & 6 & 9 & 10 & 13 &  $[\textcolor{blue}{17},\textcolor{green}{18}]$ & $[\textcolor{blue}{19},\textcolor{orange}{21}]$ & $[\textcolor{blue}{21},\textcolor{orange}{30}]$ \\
\hline
4 & 3 & 5 & 8 & 10 & $[\textcolor{blue}{12},\textcolor{green}{13}]$ & $[\textcolor{blue}{15},\textcolor{green}{16}]$ & $[\textcolor{blue}{17},\textcolor{orange}{21}]$ & $[\textcolor{blue}{21},\textcolor{orange}{25}]$ \\
\hline
5 & 3 & 5 & 8 & 10 & $[\textcolor{blue}{12},\textcolor{green}{13}]$ & $[\textcolor{blue}{15},\textcolor{green}{17}]$ & $[\textcolor{blue}{18},\textcolor{orange}{21}]$ & $[\textcolor{blue}{20},\textcolor{orange}{25}]$ \\
\hline
7 & 3 & 5 & 7 & 10 & $[\textcolor{blue}{12},\textcolor{green}{13}]$ & 14 & $[\textcolor{blue}{17},\textcolor{orange}{21}]$ & $[\textcolor{blue}{19},\textcolor{orange}{25}]$ \\
\hline
8 & 3 & 5 & 7 & 9 & $[\textcolor{blue}{12},\textcolor{green}{13}]$ & $[\textcolor{blue}{14},\textcolor{green}{15}]$ & $[\textcolor{blue}{16},\textcolor{orange}{21}]$ & $[\textcolor{blue}{19},\textcolor{orange}{25}]$ \\
\hline
9 & 3 & 5 & 7 & 9 & 12 & $[\textcolor{blue}{14},\textcolor{green}{15}]$ & $[\textcolor{blue}{16},\textcolor{orange}{21}]$ & $[\textcolor{blue}{18},\textcolor{orange}{25}]$ \\
\hline
\end{tabular}
\end{table}

For the first line of Table \ref{table2}, we refer to \cite{kurz2023divisible}, where these values are given in the context of minimal codes. 

Whenever $2k-1\leq q+1$, we may take a $[2k-1,k,k]_q$ MDS code, that is $2k-1$ points on an arc.

In dimension $2$, it is always sufficient to take $3$ distinct points.

For $q=3$, the $[6,3,3]_3$ code with generator matrix 
\[G=\begin{bmatrix}
 1&0&0&1&0&2\\
 0&1&0&2&2&1\\
 0&0&1&1&1&1
\end{bmatrix}\]
is intersecting and it is clearly the shortest (a $[5,3,\geq 3]_3$ code does not exist).

For $q=3$, there is no $[8,4,4]_3$ intersecting code by {\sc Magma} calculations, but an intersecting $[9,4,4]_3$ exists by concatenation (see Lemma \ref{lem:concatenation}). 

For $q=4$, the $[8,4,4]_4$ code with generator matrix 
\[G=\begin{bmatrix}
 1&0&0&0&0&1&1&1\\
 0&1&0&0&1&1&1&0\\
 0&0&1&0&1&0&1&\alpha\\
 0&0&0&1&0&\alpha&\alpha^2&1
\end{bmatrix}\]
(here $\alpha$ is a primitive element of $\F_4$) is intersecting and it is clearly the shortest (a $[7,4,\geq 4]_4$ code does not exist).

For $q=5$, the $[8,4,4]_5$ code with generator matrix
\[G=\begin{bmatrix}
1&0&0&0&1&0&3&4\\
0&1&0&0&4&2&4&0\\
0&0&1&0&0&4&2&4\\
0&0&0&1&4&1&3&4
\end{bmatrix}\]
is intersecting and it is clearly the shortest (a $[7,4,\geq 4]_5$ code does not exist).

For $q\in\{3,4,5,7\}$, the $[10,5,5]_q$ code with generator matrix $G_q$ equal to   
\begin{align*}
G_3=\begin{bmatrix}
1&0&0&0&0&1&2&2&2&1\\
0&1&0&0&0&1&1&1&0&1\\
0&0&1&0&0&1&1&0&2&2\\
0&0&0&1&0&2&1&2&2&0\\
0&0&0&0&1&0&2&1&2&2
\end{bmatrix}, & \qquad G_4=\begin{bmatrix}
1&0&0&0&0&\alpha&0&1&\alpha^2&\alpha\\
0&1&0&0&0&\alpha&\alpha^2&\alpha&\alpha^2&0\\
0&0&1&0&0&0&\alpha&\alpha^2&\alpha&\alpha^2\\
0&0&0&1&0&\alpha^2&\alpha&1&0&\alpha^2\\
0&0&0&0&1&\alpha^2&1&1&\alpha&1
\end{bmatrix},\\G_5=\begin{bmatrix}
1&0&0&0&0&0&1&1&1&1\\
0&1&0&0&0&1&0&2&4&3\\
0&0&1&0&0&4&2&1&4&3\\
0&0&0&1&0&4&1&2&3&4\\
0&0&0&0&1&4&3&0&1&4
\end{bmatrix}, & \qquad G_7=\begin{bmatrix}
1&0&0&0&0&0&1&1&1&1\\
0&1&0&0&0&1&3&6&6&3\\
0&0&1&0&0&3&4&2&5&1\\
0&0&0&1&0&5&5&0&4&2\\
0&0&0&0&1&6&1&6&6&0
\end{bmatrix},
\end{align*}
is intersecting  and it is clearly the shortest (a $[9,5,\geq 5]_q$ code does not exist).

For $q=3$, the $[13,6,6]_3$ code with generator matrix 
\[G=\begin{bmatrix}
1&0&0&0&0&0&2&1&1&0&0&2&2\\
0&1&0&0&0&0&2&0&2&1&0&2&1\\
0&0&1&0&0&0&1&1&2&2&1&1&0\\
0&0&0&1&0&0&0&1&1&2&2&1&1\\
0&0&0&0&1&0&1&2&0&1&2&0&2\\
0&0&0&0&0&1&2&2&0&0&1&1&2
\end{bmatrix}\]
is intersecting and it is the shorter: actually, there is no $[11,6,\geq 6]_3$ code and every $[12,6,6]_3$ is equivalent to the extended ternary Golay code (see \cite{pless1968uniqueness}), which is not intersecting.

For $q=9$, the $[12,6,6]_9$ code with generator matrix 
\[G=\begin{bmatrix}
1&0&0&0&0&0&\alpha^7&2&\alpha^6&\alpha^3&\alpha^6&0\\
0&1&0&0&0&0&2&\alpha&\alpha^7&\alpha&1&\alpha\\
0&0&1&0&0&0&2&\alpha^7&2&\alpha&\alpha^6&\alpha^2\\
0&0&0&1&0&0&\alpha^6&0&\alpha^2&1&1&\alpha\\
0&0&0&0&1&0&\alpha^3&\alpha^2&1&\alpha&\alpha^6&2\\
0&0&0&0&0&1&\alpha^3&\alpha^6&\alpha&\alpha&\alpha&\alpha^3
\end{bmatrix}\]
is intersecting and it is clearly the shortest (a $[11,6,\geq 6]_9$ code does not exist).

For $q=7$, the $[14,7,7]_7$ code with generator matrix 
\[G=\begin{bmatrix}
1&0&0&0&0&0&0&3&3&4&3&2&2&6\\
0&1&0&0&0&0&0&6&3&3&4&3&2&2\\
0&0&1&0&0&0&0&2&6&3&3&4&3&2\\
0&0&0&1&0&0&0&2&2&6&3&3&4&3\\
0&0&0&0&1&0&0&3&2&2&6&3&3&4\\
0&0&0&0&0&1&0&4&3&2&2&6&3&3\\
0&0&0&0&0&0&1&3&4&3&2&2&6&3
\end{bmatrix}\]
is intersecting and it is clearly the shortest (a $[13,7,\geq 7]_7$ code does not exist).

The lower bounds are all in blue and they follow from Theorem \ref{thm:distance} and from the database of the codes with the best known parameters in {\sc Magma}. The upper bounds in orange (that is, exactly for the columns corresponding to $k=8$ and $k=9$) may be obtained by concatenating the shortest intersecting codes over proper extensions (see Lemma \ref{lem:concatenation}). The bounds in green come from extensive research in {\sc Magma}. This has been done by starting from an $[n,k-1,d]_q$ optimal intersecting code and randomly building an $[n,k,d]_q$ code from it or simply taking the codes with the best known parameters in {\sc Magma}. 
\bigskip

\section{Small explicit constructions for large dimensions}\label{sec:costructions}

In the previous sections we provided bounds on the size of the smallest non-$2$-cohyperplanar sets in projective spaces of given dimensions, together with a non-constructive existence result. The sets presented in Example \ref{exa:arcs} meet the bound for small dimensions, as we have already observed. The aim of this section is to provide \emph{explicit} constructions of small non-$2$-cohyperplanar sets, or equivalently of short intesecting codes, for large dimensions.

\subsection{Algebraic geometry intersecting codes}

Algebraic geometry is a useful tool for constructing families of codes with good parameters.
These codes, called \emph{algebraic geometry codes}, are a generalization of Reed-Solomon codes: whereas Reed-Solomon codes are obtained from the evaluation of polynomials of bounded degree in several points of $\Fq$, algebraic geometry codes are obtained by evaluating polynomials from the Riemann-Roch space of a divisor over an algebraic curve over $\Fq$.
For a more extensive introduction to algebraic geometry codes, we refer the reader to \cite[Chapter 15]{huffman2021concise}. The family of algebraic geometry codes provides \emph{explicit} constructions of asymptotically good codes, some of which turn out to be intersecting. In many cases however, in order to obtain explicit constructions with maximum rate we need to concatenate algebraic geometry codes with well-chosen intersecting codes of low dimension.
For instance, when $q$ is a prime, there are no constructions of asymptotically good AG codes, meaning we must concatenate intersecting AG codes over some extension of $\Fq$ with suitable intersecting codes over $\Fq$.

\medskip

First we define the Singleton defect of a code.

\begin{definition}
Let $\C$ be a code with parameters $[n, k, d]_{q}$.
Its \emph{Singleton defect} is the quantity
$$\Delta = 1 - \frac{k+d}{n+1}.$$
\end{definition}

Notice that a code is MDS if and only if $\Delta = 0$, while a code with ``bad'' parameters has a large Singleton defect.

\begin{definition}
    The \emph{Ihara constant} of $\F_q$ is
$$A(q)=\limsup_{g(X)\to\infty}\frac{n(X)}{g(X)},$$
where $X$ ranges over all curves over $\F_q$, $n(X)=|X(\F_q)|$ is the number of rational points of $X$ and $g(X)$ is the genus of $X$.
\end{definition}

The best possible Singleton defect attainable by AG codes is $A(q)^{-1}$ (\cite[Chapter 15, Corollary 15.3.14]{huffman2021concise}). Note also that, provided the Singleton defect is at least $A(q)^{-1}$, any choice of parameters $R$ and $\delta$ that sum to $1 - A(q)^{-1}$ is attainable.
The Drinfeld-Vladut bound \cite[Theorem 2.3.22]{MR1186841} states that $A(q) \leq \sqrt{q} - 1$, which gives a lower bound on the best possible Singleton defect reachable by AG codes.
Nevertheless, there exist explicit constructions of AG codes with Singleton defect close to this lower bound.
Most notably, when $q$ is a square it is possible to reach the Drinfeld-Vladut bound, as first proved in \cite{Ihara1982SomeRO}.

\medskip

In \cite{randriambololona20132}, the author establishes the following theorem:

\begin{theorem}[Theorem 2, \cite{randriambololona20132}]\label{thm:randriam}
Suppose that $A(q) \geq 4$. Then there exists an explicit family of asymptotically good intersecting codes with asymptotic rate
$$R = \frac{1}{2} - \frac{1}{2A(q)}.$$
\end{theorem}

\begin{remark}
The proof of Theorem \ref{thm:randriam} relies on non-trivial algebraic geometric arguments. A simpler way to construct intersecting AG codes would be to consider families of AG codes with $\delta > 1/2$, which are intersecting, by Lemma \ref{lem:ABinter}.
The best possible rate using this method is
$$R = \frac{1}{2} - \frac{1}{A(q)}.$$
Hence, Theorem \ref{thm:randriam} is an improvement over this simpler method.
\end{remark}

Theorem \ref{thm:randriam} often yields the best-known explicit constructions of intersecting codes over $\Fq$. However, when $q$ is small, or a prime, $A(q) \leq 4$. In these cases, it is therefore necessary to use concatenation in order to construct explicit sequences of intersecting AG codes of short length.

\begin{remark}\label{rem3}
Before diving into the details, we make one more observation.
The highest rate of a non-trivial intersecting code is attained by a code with parameters $[3, 2, 2]_q$ (over any base field), corresponding to three distinct points on the projective line.
Moreover, an intersecting AG code constructed with the above theorem must have rate lower than $1/2$. 
This means that any non-trivial concatenation of an intersecting code with intersecting AG codes must have a rate of at most $1/3$.
Consequently, if over some field $\Fq$ there are intersecting AG codes that have rate larger than $1/3$, there is no construction involving concatenation that will yield a better rate.
\end{remark}

\begin{theorem}\label{agthm}
The following upper bounds, which stem from \emph{explicit} constructions involving (possibly concatenated) AG codes, hold:
\begin{itemize}
\item if $q$ is a square and $q \geq 25$, then
$$\limsup_{k \rightarrow \infty} \frac{i(k, q)}{k} \leq 2 + \frac{2}{\sqrt{q} - 2};$$
\item if $q = p^{2m+1}$ is an uneven power of a prime (but not a prime) and $q \geq 32$, then
$$\limsup_{k \rightarrow \infty} \frac{i(k, q)}{k} \leq \frac{4}{2 - \frac{1}{p^{m}-1} - \frac{1}{p^{m+1}-1}},$$
\item If $q$ is a prime and $q \geq 11$, then
$$\limsup_{k \rightarrow \infty} \frac{i(k, q)}{k}\leq 3 + \frac{3}{q - 2}.$$
\end{itemize}
For the remaining values of $q$, Table \ref{table3} provides the upper bounds obtained by concatenating AG codes with suitable intersecting codes.
\begin{table}[ht!]
\centering
\caption{Upper bounds obtained with AG codes, for exceptional values of $q$}
\label{table3}
\begin{tabular}{cccc}
$q$  & Parameters of inner code & \begin{tabular}[c]{@{}c@{}}Upper bound for\\ $\limsup_{k \rightarrow \infty} i(k, q)/k$\end{tabular} & Probabilistic bound \\ \hline
$2$  & $[15, 6]_{2}$            & $5.8334$                                                                                             & $4.8189$            \\
$3$  & $[10, 5]_{3}$            & $4.3561$                                                                                             & $3.7382$            \\
$4$  & $[5, 3]_{4}$             & $4.1667$                                                                                             & $3.3539$            \\
$5$  & $[5, 3]_{5}$             & $3.9025$                                                                                             & $3.1507$            \\
$7$  & $[7, 4]_{7}$             & $3.5745$                                                                                             & $2.9331$            \\
$8$  & $[3, 2]_{8}$             & $3.5$                                                                                                & $2.8666$            \\
$9$  & $[3, 2]_{9}$             & $3.4286$                                                                                             & $2.8148$            \\
$16$ & $[3, 2]_{16}$            & $3.2143$                                                                                             & $2.6266$            \\
$27$ & $[3, 2]_{27}$            & $3.12$                                                                                               & $2.5146$           
\end{tabular}
\end{table}
\end{theorem}

\begin{proof}
In all three cases we call $R_{q}$ the largest rate reached by AG codes over $\Fq$.
Recall that we will then obtain the asymptotic upper bound
$$\limsup_{k \rightarrow \infty} \frac{i(k, q)}{k} \leq R_{q}^{-1}.$$

\begin{itemize}
    \item As we have already mentioned, when $q$ is a square, there are explicit constructions of AG codes reaching the Drinfeld-Vladut bound, that is, with Singleton defect equal to $(\sqrt{q} - 1)^{-1}$. The best possible rate yielded by Theorem \ref{thm:randriam} is
$$R_{q} = \frac{1}{2} - \frac{1}{2(\sqrt{q}-1)}.$$
For the hypotheses of Theorem \ref{thm:randriam} to be verified, we must have $q \geq 25$.
Note that Remark~\ref{rem3} tells that we do not need to concatenate, since as soon as $q \geq 25$, we have $R \geq 3/8 \geq 1/3$.

    \item When $q = p^{2m+1}$ (with $m\geq 1$), a prominent result shown in \cite{garciastichtnoth2015} provides an explicit construction of AG codes satisfying the lower bound 
$$A(q) \geq 2 \left( \frac{1}{p^{m}-1} + \frac{1}{p^{m+1}-1} \right)^{-1}.$$
In order to satisfy the hypothesis of Theorem \ref{thm:randriam}, we need $A(q) \geq 4$.
This is the case for every value of $q$ except $q = 8$ and $q = 27$.
Again, in this case Remark \ref{rem3} tells that we do not need to concatenate.

\item If $q$ is a prime, there are no explicit constructions of asymptotically good intersecting AG codes. Hence we need to concatenate. When $q \geq 11$ it is straightforward to check that concatenating AG codes over $\F_{q^2}$ whose parameters meet the Drinfeld-Vladut bound with a $[3, 2, 2]_q$ code will always produce the shortest explicit construction. This yields
$$R_{q} \geq \frac{2}{3} \cdot \left( \frac{1}{2} - \frac{1}{2(q-1)} \right) = \frac{1}{3} - \frac{1}{3(q-1)}.$$
\end{itemize}
The remaining cases are obtained by concatenation, in each case using Theorem \ref{thm:randriam} and the best known lower bounds on $A(q)$ to obtain the best possible rate for the outer code.
\end{proof}

\begin{remark}
[Comparison with the probabilistic bound of Theorem \ref{prob}]
The bound from Theorem \ref{agthm} outperforms the probabilistic bound exactly in the following cases:
\begin{itemize}
    \item if $q\geq 49$ is a square;
    \item if $q\geq 128$ is an odd power of a prime.
\end{itemize}
Hence the asymptotic upper bound provided by Theorem \ref{agthm} is best for almost all non-prime $q$.
Moreover, this is a \emph{constructive} bound: the codes that reach it can be explicitly constructed in polynomial time, as noted in \cite{randriambololona20132}.    
\end{remark}

\begin{remark}
Recall that in the binary case intersecting codes coincide with minimal codes (see Lemma \ref{lem:intermin}).
There exist numerous short constructions of minimal codes, for instance in \cite{bartoli2023small,alon2023strong,CZ}.
To the best of our knowledge, the shortest explicit construction was given in \cite{CZ}.
The construction recorded in Table \ref{table3} provides an improvement and it is, to the best of our knowledge, the shortest \emph{explicit} construction of minimal codes over $\mathbb{F}_{2}$.
\end{remark}

\medskip

\subsection{A construction using expander graphs}

In this subsection we provide an explicit construction of non-2-cohyperplanar sets using expander graphs, based on the approach used in \cite{alon2023strong} for strong blocking sets. Even though the resulting construction will be longer than the one in the previous subsection, we believe that it is still interesting because it provides a geometric insight into non-2-cohyperplanar sets, as well as a link with other well-known combinatorial and geometric objects. We will construct small sets of lines with avoidance property and we will take $3$ points on each line, obtaining a small non-$2$-cohyperplanar set, by Proposition \ref{pro:3points}.

\begin{definition}
Let $\mathcal{G}=(V,E)$ be a graph with $n$ vertices, say $V=\{u_1,\ldots,u_n\}$. The \emph{adjacency matrix}  $A_\mathcal{G}$ of  $\mathcal{G}$ is the $n\times n$ matrix with coefficients $a_{i,j}=|\{\text{edges connecting } u_i\text{ and } u_j\}|$.  
\end{definition}

The matrix is clearly diagonalizable over the real field (it is symmetric). Let us call $\lambda_1(\mathcal{G})\geq \lambda_2(\mathcal{G})\geq \ldots \geq \lambda_n(\mathcal{G})$ its eigenvalues. Recall that if $\mathcal{G}$ is $t$-regular, then $\lambda_{1}(\mathcal{G}) = t$.
We also define
$$\lambda(\mathcal{G}) = \max\{|\lambda_2(\mathcal{G})|, \dots, |\lambda_n(\mathcal{G})|\}.$$

\begin{definition}
An \emph{$(n, t, \lambda)$-graph} $\mathcal{G}$ is a $t$-regular graph with $n$ vertices such that 
$|\lambda(\mathcal{G})| \leq \lambda$.\\
A $t$-regular graph $\mathcal{G}$  with $\lambda(\mathcal{G}) \leq 2\sqrt{t-1}$ is called \emph{Ramanujan graph}.
\end{definition}

\begin{theorem}[Alon-Bopanna] \label{alonbopanna}
For an $(n, t, \lambda)$-graph, 
$$\lambda \geq 2\sqrt{t-1} - o(1)$$
as $n \rightarrow \infty$.
\end{theorem}

In \cite{nogaexpander}, the author proves the following.
\begin{theorem}[Theorem 1.3, \cite{nogaexpander}] \label{explicitexpander}
For every degree $t$, every $\varepsilon$ and all sufficiently large $n \geq n_{0}(t, \varepsilon)$, where $nt$ is even, there is an explicit construction of an $(n, t, \lambda)$-graph with
$$\lambda \leq 2\sqrt{t-1} + \varepsilon.$$
\end{theorem}

The following is an invariant of graphs, which will be fundamental to get our construction.

\begin{definition}
Let $\mathcal{G} = (V, E)$ be a simple connected graph. For any subgraph $\mathcal{H}$, let $\kappa(\mathcal{H})$
denote the largest size of a connected component in $\mathcal{H}$. The \emph{integrity} of $\mathcal{G}$ is the integer
$$\iota(\mathcal{G})=\min\{|S|+\kappa(\mathcal{G}-S)\mid S\subseteq V\}.$$
\end{definition}

\begin{proposition}[Corollary 3.4, \cite{alon2023strong}]
For an $(n, t, \lambda)$-graph $\mathcal{G}$,
$$\iota(\mathcal{G}) \geq n \cdot \frac{t-\lambda}{t+\lambda}.$$    
\end{proposition}

The next result, proved in \cite{alon2023strong}, is the link between the theory of expander graphs and lines with the avoidance property, and then with strong blocking sets and non-2-cohyperplanar sets.

\begin{proposition}[Lemma 4.4, \cite{alon2023strong}] \label{propexp}
Let $\mathcal{M}=\{P_1,\ldots,P_n\}\subseteq \PG(k-1,q)$ be a projective $[n,k,d]_q$ system and $\mathcal{G}=(\mathcal{M},E)$ a graph. If $$\iota(\mathcal{G})\geq n-d+1,$$ 
then  the set of lines 
$$\mathcal{L}(\mathcal{M},\mathcal{G})=\{\langle P_i,P_j\rangle \mid P_iP_j\in E \}$$
satisfies the avoidance property.\\
Hence, if $\mathcal{G}$ is an $(n,t,\lambda)$-graph and
$$\frac{t-\lambda}{t+\lambda}\geq 1-\delta+\frac{1}{n},$$
then $\mathcal{L}(\mathcal{M},\mathcal{G})$ satisfies the avoidance property.
\end{proposition}

Combining Theorem \ref{alonbopanna}, Theorem \ref{explicitexpander}, Proposition \ref{propexp} and Proposition \ref{pro:3points}, we obtain the following result.

\begin{theorem}\label{intersectingexpanderconstruction}
Assume that there is an explicit construction of projective $[n, Rn, \delta n]_{q}$ systems and an integer $t$ such that
$$\frac{t-2\sqrt{t-1}}{t+2\sqrt{t-1}} > 1 - \delta.$$
Then there exist \emph{explicit} families of non-2-cohyperplanar sets with size tending to
$$\left(1+ \frac{t}{2}\right)n,$$
as $n \rightarrow \infty$. 
\end{theorem}

\begin{proof}
Let $\varepsilon > 0$ and choose $n \geq n_{0}(t, \varepsilon)$ such that $nt$ is even.
We call $\mathcal{M}$ the projective $[n, Rn, \delta n]_{q}$ system.
By Theorem \ref{explicitexpander}, there is an explicit construction of a $(n, t, \lambda)$-graph with $\lambda = 2\sqrt{t-1} + \varepsilon$.
Let us call this graph $\mathcal{G}_{n, t} = (V_{n, t}, E_{n, t})$.
Notice that by our assumptions it is possible to choose $\varepsilon$ small enough and $n$ large enough so that
$$\frac{t-\lambda}{t+\lambda} \geq 1 - \delta + \frac{1}{n}.$$
Therefore, $\mathcal{L}(\mathcal{M},\mathcal{G}_{n, t})$ satisfies the avoidance property.

By Proposition \ref{pro:3points}, if we choose $3$ points on every line of $\mathcal{L}(\mathcal{M},\mathcal{G}_{n, t})$, we get a non-2-cohyperplanar set. In order to get the smallest such set, we choose every vertex and one point (different from the vertices) on every edge of $\mathcal{G}_{n, t}$.
This yields $n + nt/2$ points.
\end{proof}

Let us give one example of an application of Theorem \ref{intersectingexpanderconstruction}.
For simplicity's sake, we will consider only the case when $q$ is a square, since this yields the best possible AG codes as well as the simplest formula for the Singleton defect (which is nice for computations). Consider a family of AG $[n,Rn,\delta n]_q$ codes such that 
$$R+\delta=1-\frac{1}{\sqrt{q}-1}.$$
Note that the size of the non-2-cohyperplanar set that we may obtain is
$$\left(1+ \frac{t}{2}\right)n = \frac{1+ t/2}{R} \cdot k,$$
where $k$ is the dimension of the AG code. According to Theorem~\ref{intersectingexpanderconstruction}, we need $t$ verifying
$$\frac{t-2\sqrt{t-1}}{t+2\sqrt{t-1}} > 1 - \delta = R + \frac{1}{\sqrt{q} - 1}.$$
Setting
$$R(q, t) = \frac{t-2\sqrt{t-1}}{t+2\sqrt{t-1}} - \frac{1}{\sqrt{q} - 1}$$
and
$$\alpha(q, t) = \frac{1 + t/2}{R(q, t)},$$
we want to minimize the value of $\alpha(q,t)$. Notice that since $t$ is the degree of a vertex, $t$ has to be an integer, which rather limits the possibilities for optimization for a given $q$.

When $q \rightarrow \infty$, the second term in the expression of $R(q, t)$ vanishes. This yields an expression of $\alpha(q, t)$ which does not depend on $q$, for which it is easy to check that the minimum value is reached for $t=10$.
Hence we get explicit constructions of non-2-cohyperplanar sets with
$$R(q,10) = \frac{1}{4} - \frac{1}{\sqrt{q} - 1}$$
and
$$\alpha(q,10) = \frac{6}{R(q,10)}\to 24,$$
as $q \rightarrow \infty$.

By computing the value of $\alpha(q, t)$ for integer values of $t$, it is possible to verify that, for $q\geq 89^{2}$, $t=10$ gives the minimum value for $\alpha(q,t)$. For smaller values of $q$, the best values of $t$ and of $\alpha(q,t)$ are reported in Table \ref{table_expander}. 
For $q=4$ the Singleton defect is $1$ so $R > 0$ is impossible, meaning that our construction does not work.

\begin{table}[ht!]
\caption{Smallest values of $\alpha(q,t)$ for small square $q$}
\label{table_expander}
\begin{tabular}{c|c|c}
$q$                                 &  $t$ & $\alpha(q,t)$ \\ \hline
$3^{2}$&86&299.5378\\
$4^{2}$&39&110.0490\\
$5^{2}$&27&71.8927\\
$7^{2}$&20&48.6300\\
$8^2$&18&43.7121\\
$9^2$&17&40.4255\\
$11^2$&15&36.2747\\
$13^2$&14&33.7937\\
$16^2$&13&31.5103\\
$17^2\leq q\leq 19^2$&13&$\sim 30$\\
$23^2\leq q\leq 27^2$&12&$\sim 28$\\
$29^2$&12&27.7441\\
$31^2\leq q\leq 32^2$&11&$\sim 27$\\
$37^2\leq q\leq 49^2$&11&$\sim 26$\\
$53^2\leq q\leq 83^2$&11&$\sim 25$\\
\end{tabular}
\end{table}

\medskip

\section{On the 2-wise weighted Davenport constants}\label{sec:additive}

In this section we investigate links between intersecting codes and zero-sum problems over finite abelian groups, in particlar with generalizations of the Davenport constant. Zero-sum problems over finite abelian groups have been studied since the 1960s, and the Davenport constant and its generalizations remain a main subject in this area (see \cite{GAO2006337} for a general survey on the topic). The coding theoretical approach to problems about zero-sum subsequences in finite abelian groups is not new (see for example \cite{MacWilliams1977TheTO,cohen1999subset}) and an investigation of the Davenport constant with these methods has been already done in \cite{SP,multiwise}. However, our framework is more general: we consider intersecting codes over any finite field and their relation with weighted Davenport constants.

\subsection{The general setting}

Let $G$ be a finite abelian group.

\begin{definition}
Let $a_{1}, \dots, a_{n} \in G$ be a finite sequence of elements of $G$. For such a sequence, we define a \emph{zero-sum subsequence} as a sequence $a_{i_{1}}, \dots, a_{i_{r}}$, with $\{i_1,\ldots,i_r\}\subseteq \{1,\ldots,n\}$, verifying $\sum_{k = 1}^{r} a_{i_{k}} = 0$.  
\end{definition}

If a sequence is long enough, then necessarily it admits a zero-sum subsequence. Therefore, it makes sense to ask from which threshold this occurs for all sequences.

\begin{definition}
The \emph{Davenport constant} of $G$, noted $\D(G)$, is the smallest integer $\ell$ such that every sequence of $\ell$ elements of $G$ has a zero-sum subsequence.\\
The quantity $\dd(G)$ is the largest integer $\ell$ such that there is a sequence of length $\ell$ with no zero-sum subsequences.
\end{definition}

\begin{remark}
The quantity $\dd(G)$ is sometimes also referred to as the small Davenport constant in the literature. This will also be the case in the present article.
One has
$$\D(G) = \dd(G) +1.$$
\end{remark}

\begin{remark}
Another definition of the Davenport constant (which is equivalent in this setting) goes as follows: consider only zero-sum sequences of $G$, that is sequences $a_{1}, \dots, a_{n} \in G$ such that $\sum_{i = 1}^{n} a_{i} = 0_{G}$.
Then $\D(G)$ is the length of the longest such sequence that does not split into $2$ disjoint non-trivial zero-sum subsequences. Indeed, considering a zero-sum sequence of length $n > \D(G)$, it is possible to take the first $\D(G)$ terms, among which there will be a zero-sum subsequence by definition. Its complement must be a zero-sum subsequence as well. Therefore the whole zero-sum sequence splits into $2$ disjoint zero-sum subsequences. Conversely, there are sequences of length $\dd(G)$ with no zero-sum subsequence, meaning that by adding one last element in order to form a zero-sum sequence of length exactly $\D(G) = \dd(G) + 1$ we get a sequence which clearly does not have $2$ disjoint zero-sum subsequences.
\end{remark}

\begin{example}
Let $C_{n}$ be the (additive) cyclic group with $n$ elements and let $1$ denote a generating element.
The sequence $a_{1} = 1, \dots, a_{n-1} = 1$ is a sequence of length $n-1$ with no zero-sum subsequence.
Therefore $\dd(C_{n}) \geq n-1$, which implies $\D(C_{n}) \geq n$. It is well-known that $\D(G) \leq |G|.$ 
Hence, $\D(C_{n}) = n$.
\end{example}

The Davenport constant has been studied intensively, and in fact it has been generalized in a number of ways, two of which we present and use here.

\medskip

Let us define first the multiwise Davenport constants, a generalization introduced by Halter-Koch in \cite{HalterKoch1992AGO}.

\begin{definition}
Let $a_{1}, \dots, a_{n} \in G$ be a finite sequence of elements of $G$ and $j$ be a positive integer. We say that $j$ zero-sum subsequences are \emph{disjoint} if their indices belongs to $j$  disjoint subsets of $\{1, \dots, n\}$.

The $j$-\emph{wise Davenport constant} $\D_{j}(G)$ is the smallest integer $\ell$ such that every sequence of $\ell$ elements of $G$ has $j$ disjoint zero-sum subsequences.
\end{definition}

The usual Davenport constant is $\D_{1}(G) = \D(G)$.
Clearly, one has
\begin{equation}\label{davenportineq}
\D(G) \leq \D_{j}(G) \leq j\D(G).
\end{equation}

The second generalization arises when considering weighted zero-sum subsequences. There are various natural ways to introduce weights in these types of problems. The one that we recall below  received considerable attention in the last two decades since the work of Adhikari et al. \cite{Adhikari2006_1,Adhikari2006_2}. 

\begin{definition}
Let $a_{1}, \dots, a_{n} \in G$ be a finite sequence of elements of $G$ and let $$\emptyset \neq W \subseteq \{0, 1, \dots, \exp(G)-1\}.$$ 
A \emph{$W$-weighted zero-sum subsequence} is a sequence $a_{i_{1}}, \dots, a_{i_{r}}$, with $\{i_1,\ldots,i_r\} \subseteq \{1,\ldots,n\}$, verifying 
$$\sum_{i = 1}^{r} \varepsilon(i)a_{j_i}=0$$
for some $\varepsilon:\N\to W$.
In case $W = \{1, \dots, \exp(G)-1\}$ the sequence is called a \emph{fully-weighted zero-sum subsequence}. 

The \emph{$W$-weighted Davenport constant} $\D^W(G)$ is the smallest integer $\ell$ such that every sequence of $\ell$ elements of $G$ has a $W$-weighted zero-sum subsequence.

The \emph{fully-weighted Davenport constant} $\D^f(G)$ is the smallest integer $\ell$ such that every sequence of $\ell$ elements of $G$ has a fully-weighted zero-sum subsequence.
\end{definition}
Instead of considering subsets of $\{0, 1, \dots, \exp(G)-1\}$ one could also consider subsets of the integers, yet this is essentially equivalent. 

It is also possible to examine multiwise weighted Davenport constants.

\begin{definition}
The $j$-\emph{wise $W$-weighted Davenport constant} $\D_{j}^{W}(G)$ is the smallest integer $\ell$ such that every sequence of length $\ell$ of $G$ has $j$ disjoint $W$-weighted zero-sum subsequences. The $j$-\emph{wise fully-weighted Davenport constant} $\D_{j}^{f}(G)$ is $\D_{j}^{W}(G)$ with $W=\{1, \dots, \exp(G)-1\}$.    
\end{definition}

When $j=2$, there is a relation between this last constant and intersecting codes over prime fields, as remarked  in \cite{SP,multiwise}. We will explain this link in a more general scenario (see Theorem~\ref{thm:link}), including intersecting codes over any finite field.

\medskip

\subsection{Our generalization}

We are now ready to properly define our generalization of the fully-weighted Davenport constant. Let us first recall the definition of $\mathcal{W}$-weighted Davenport constant where the set of weights $\mathcal{W}$ is defined as a non-empty subset of endomorphisms of $G$. This was first introduced in \cite{zeng2011weighted}. The interested reader may also refer to \cite{grynkiewicz2013structural}.

\begin{definition}
Let $G$ be a finite abelian group and let $\mathcal{W}$ be a non-empty set of group endomorphisms of $G$, which we call a  \emph{set of weights} for $G$. Let $a_{1}, \dots, a_{n} \in G$ be a finite sequence of elements of $G$ and $\mathcal{W}$ a set of weights for $G$. A $\mathcal{W}$\emph{-weighted zero-sum subsequence} is a sequence $a_{i_{1}}, \dots, a_{i_{r}}$, with $\{i_1,\ldots,i_r\}\subseteq \{1,\ldots,n\}$, verifying 
$$\sum_{i = 1}^{r} \varepsilon_{i}(a_{j_i})=0$$
for some $\varepsilon_{i}\in \mathcal{W}$.

Let $j$ be a positive integer. The \emph{$j$-wise $\mathcal{W}$-weighted Davenport constant} $\D_j^\mathcal{W}(G)$ is the smallest integer $\ell$ such that any sequence of length $\ell$ of $G$ has $j$ disjoint $\mathcal{W}$-weighted zero-sum subsequences.
 Similarly, the \emph{$j$-wise $\mathcal{W}$-weighted small Davenport constant} $\dd_{j}^{\mathcal{W}}(G)$ is the largest integer $\ell$ such that there exists a sequence of length $\ell$ of $G$ that does not have  $j$ disjoint $\mathcal{W}$-weighted zero-sum subsequences.
\end{definition}

When $G$ is an elementary $p$-group, that is, when
$$G=E_{p^{hr}}=\underbrace{C_p\oplus \ldots\oplus C_p}_{hr\text{ times}},$$
the elementary abelian group of order $p^{hr}$ (here $p$ is a prime and $h$ and $r$ are positive integers), it is possible to consider a group isomorphism $E_{p^{hr}}\cong \mathbb{F}_q^r$, where $q=p^h$. Clearly, this concerns only the additive part. Below we use the multiplicative structure on $\mathbb{F}_q^r$ to introduce a set of weights that can be seen as a generalization of fully-weighted for elementary abelian groups.

\begin{definition}
For $G=E_{p^{hr}}$ an elementary abelian group of order $p^{hr}$, consider a group isomorphism $\varphi : E_{p^{hr}} \to \Fq^{r}$. Define
$$\mathcal{Q}_{h} = \{m_x:y\mapsto \varphi^{-1}(x\varphi(y)) \in {\rm End}(E_{p^{hr}}) \mid x \in \Fq \}$$
the set of weights induced by the scalar multiplication of $\F_q=\F_{p^h}$.
\end{definition}

While the sets of weights $\mathcal{Q}_{h}$ depend in principle on our choice of isomorphism $\varphi$, it is easy to see that the value of the associated Davenport constants does not depend on this choice. This is why we do not include $\varphi$ in the notation of $\mathcal{Q}_{h}$.

Below, we study the $j$-wise $\mathcal{Q}_h$-weighted Davenport constant $\D_j^{\mathcal{Q}_{h}}(E_{p^{hr}})$. In order to simplify the notation, we will denote it by $\D_j^h(E_{p^{hr}})$.

\begin{remark}\label{linear algebra}
For an elementary abelian group $E_{p^{hr}}$ of cardinality $p^{hr}$, let us underline that $\D_j^h(E_{p^{hr}})=\D_j^f(E_{p^{hr}})$ if $h=1$ and that $\D_j^h(E_{p^{hr}})=\D_j(E_{p^{hr}})$ if $p=2$ and $h=1$.
Furthermore, when $j=1$, note that $D_{j}^{h}(E_{p^{hr}}) = r+1$ for elementary reasons of linear algebra over $\Fq$.
\end{remark}

The following theorem establishes the main link between these objects and the theory of intersecting codes.

\begin{theorem}\label{thm:link}
Let $E_{p^{hr}}$ be an elementary abelian group of order $p^{hr}$, where $p$ is a prime and $h,r$ are positive integers. Then $\D_{2}^{h}(E_{p^{hr}})$ is the smallest integer $n$ such that all $[n, n-r]_{p^h}$ codes are not intersecting. Therefore
$$\D_{2}^{h}(E_{p^{hr}}) = \min \{m \geq r+1 \mid m < i(m-r, p^h) \}.$$
\end{theorem}

\begin{proof}
Let $n=\D_2^h(E_{p^{hr}})-1$. By definition, there exists a sequence $a_{1}, \dots, a_{n} \in E_{p^{hr}}$ that does not admit two disjoint weighted zero-sum subsequences. Note that $n$ must be greater than $r$ by the above remark \ref{linear algebra}.
Via the isomorphism $E_{p^{hr}} \cong \F_{p^h}^{r}$, every $a_{i}$ can be seen as a (column) vector. 
Let $H$ be the matrix defined as
$$H = 
\left[
\begin{array}{c|c|c}
\, & \, & \, \\
a_{1} & \cdots & a_{n} \\
\, & \, & \,
\end{array}
\right].$$
The matrix $H$ is full-rank, because otherwise there would be a sequence of length $\D_2^h(E_{p^{hr}})$ that does not admit two disjoint weighted zero-sum subsequences, contradicting the definition: simply consider $b \notin \langle a_{1}, \dots, a_{n}\rangle$ and the prolonged sequence $a_{1}, \dots, a_{n}, b$.

Let $\C$ be the $[n,n-r]_{p^h}$ code defined by the parity-check matrix $H$. A codeword of $\C$ corresponds to a $\mathcal{W}$-weighted zero-sum subsequence of $a_{1}, \dots, a_{n}$. By the above assumption, $\C$ is then an intersecting code. Therefore 
$$n \geq i(n-r, p^h).$$
Hence 
$$\D_{2}^{h}(E_{p^{hr}}) = \max \{n>r \mid n \geq i(n-r, p^h) \}+1.$$
or, equivalently, the thesis.
\end{proof}

\begin{example}\label{exa:16}
Consider the elementary abelian group $E_{16}$ of order $16$. Let $h=1$ and $r=4$. The set $\{m \geq 5 \mid m < i(m-4, 2) \}=\{8,9,\ldots\}$ (see Table \ref{table2}), so that $\D_2(E_{16})=8$. On the other hand, if $h=2$ and $r=2$, the set 
$\{m \geq 3 \mid m < i(m-2, 4) \}=\{6,7,\ldots\}$  (see again Table \ref{table2}), so that $\D_2^2(E_{16})=6$.
\end{example}

\begin{example}\label{exa:1024}
Consider the elementary abelian group $E_{1024}$ of order $1024$. The set $\{m \geq 11 \mid m < i(m-10, 2) \}=\{17,18,\ldots\}$ (see Table \ref{table2}), so that $\D_2(E_{1024})=17$. On the other hand, if $h=2$ and $r=5$, the set $\{m \geq 6 \mid m < i(m-5, 4) \}=\{11,12,\ldots\}$ (see again Table \ref{table2}), so that $\D^2_2(E_{1024})=11$.
\end{example}

\begin{remark}
Below we record the values of $(r,\D_2(E_{2^r}))$ deduced from Table \ref{table2}:
$(1, 4)$, $(2, 5)$, $(3, 7)$, $(4, 8)$, $(5, 10)$, $(6, 11)$, $(7, 12)$, $(8,14)$, $(9,16)$, $(10,17)$, $(11,18)$, $(12,19)$, $(13,21)$, $(14,22)$, $(15,23)$, $(16,25)$, $(17,27)$. Any further improvement on the knowledge of $i(k,2)$ would allow to extend this list.
\end{remark}

\subsection{Asymptotic bounds}

Let us fix a prime $p$ and a positive integer $h$. Let us denote $q=p^h$.  
We investigate the asymptotic behavior of $\D_{2}^{h}(E_{p^{hr}})$ when $r$ grows.

\begin{lemma}\label{equivalence}
Let $\alpha \leq \liminf_{k \rightarrow \infty} i(k, p^h)/k$, and $\beta \geq \limsup_{k \rightarrow \infty} i(k, p^h)/k$.
Then
$$\limsup_{r \rightarrow \infty} \frac{\D_{2}^{h}(E_{p^{hr}})}{r} \leq \frac{\alpha}{\alpha - 1}$$
and
$$\liminf_{r \rightarrow \infty} \frac{\D_{2}^{h}(E_{p^{hr}})}{r} \geq \frac{\beta}{\beta - 1}.$$
\end{lemma}

\begin{proof}
Let $\varepsilon > 0$ and let $r$ be large enough so that, for all $m \geq r+1$, one has both 
$$i(m-r, p^h) \leq (\beta+\varepsilon)\cdot(m-r) \text{ and } i(m-r, p^h) \geq (\alpha - \varepsilon) \cdot (m-r).$$
Recall that, by Theorem \ref{thm:link}, 
$$\D_{2}^{h}(E_{p^{hr}}) = \min \{m \geq r+1 \mid m < i(m-r, q) \}.$$
Hence $\D_{2}^{h}(E_{p^{hr}}) < i(\D_{2}^{h}(E_{p^{hr}})-r, q) \leq (\beta + \varepsilon) \cdot (\D_{2}^{h}(E_{p^{hr}})-r)$, from which we obtain
$$\D_{2}^{h}(E_{p^{hr}}) \geq r \cdot \frac{\beta+\varepsilon}{\beta -1+\varepsilon}.$$
The other bound follows similarly by considering that $$\D_{2}^{h}(E_{p^{hr}})-1 \geq i(\D_{2}^{h}(E_{p^{hr}})-1-r, q)\geq (\alpha-\varepsilon)\cdot(\D_2^h(E_{p^{hr}})-1-r),$$
from which we obtain
$$\D_{2}^{h}(E_{p^{hr}})\leq 1+r\cdot \frac{\alpha-\varepsilon}{\alpha-1-\varepsilon}.$$
\end{proof}

Using the above lemma we can transform our asymptotic upper and lower bounds from the previous sections into respectively lower and upper bounds for the asymptotic value of $\D_{2}^{h}(E_{p^{hr}})$, as well as give the length of constructions of long sequences with no zero-sum subsequence.

\begin{theorem}\label{dupper}
For all prime $p$ and positive integer $h$, one has
$$\limsup_{r \rightarrow \infty} \frac{\D_{2}^{h}(E_{p^{hr}})}{r} \leq 2-\frac{1}{p^h}.$$
Moreover, for $p^h\leq 17$, this bound is improved in Table \ref{table4}.
\end{theorem}

\begin{center}
\begin{table}[ht!]
\caption{Upper bound on the asymptotic $2$-wise weighted Davenport constant}
\label{table4}
\begin{tabular}{c|c|c}
$p$ & $h$ & $\limsup_{r \rightarrow \infty} \frac{\D_{2}^{h}(E_{p^{hr}})}{r}$ \\ \hline
2 & 1   & 1.3956                                        
    \\
2 & 2  & 1.6364 
\\
2 &  3 & 1.8057                                            \\
2 & 4  & 1.9257                                            \\
3 & 1   & 1.5472                                           \\
3 & 2   & 1.8291                                            \\
5 & 1   & 1.6972                                            \\
7 & 1   & 1.7774                                           \\
11 & 1  & 1.8660                                            \\
13 & 1  & 1.8940                                            \\

17 & 1  & 1.9343                                            \\
\end{tabular}
\end{table}
\end{center}

\begin{proof}
Simply apply Lemma \ref{equivalence} using $\alpha=2+\frac{1}{p^h-1}$ from Theorem \ref{thm:lowerasymp}. As observed, $\alpha$ may be improved by looking at the Table \ref{table1} and applying Lemma \ref{equivalence}.
\end{proof}

\begin{theorem}\label{dlower1}
For every prime $p$ and every positive integer $h$, the following holds:
\begin{itemize}
\item if $h=1$ or $p^h \in \{4, 8, 9, 16, 25, 27, 32, 125\}$,
$$\liminf_{r \rightarrow \infty} \frac{\D_{2}^{h}(E_{p^{hr}})}{r} \geq \frac{2}{\log_{p^h}(2p^h-1)};$$
\item if $h=2m$ is even and $p^h \notin \{4, 9, 16, 25\}$, then
$$\liminf_{r \rightarrow \infty} \frac{\D_{2}^{h}(E_{p^{hr}})}{r} \geq 2 - \frac{2}{p^{m}};$$
\item if $h = 2m+1$ is odd and $p^h \notin \{8, 27, 32, 125\}$, then
$$\liminf_{r \rightarrow \infty} \frac{\D_{2}^{h}(E_{p^{hr}})}{r} \geq 2 - \frac{2}{2 \frac{(p^{m} - 1) (p^{m+1}-1)}{(p^{m+1} + p^{m} -2)} + 1}.$$
\end{itemize}
\end{theorem}

\begin{proof}
Simply apply Lemma \ref{equivalence} using $\beta$ from Theorem \ref{prob} and Theorem \ref{agthm}.
\end{proof}

\begin{remark}
Whenever $h=1$, we obtain the same asymptotic upper and lower bounds as in \cite{multiwise}.    
\end{remark}

\begin{remark}
Notice that Theorem \ref{dupper} and Theorem \ref{dlower1} can be considered to be an improvement on \eqref{davenportineq} for $j=2$.
Indeed, notice that \eqref{davenportineq} also holds for weighted versions of the Davenport constant.
Moreover, since $D_{1}^{h}(E_{p^{hr}}) = r+1$, as stated in Remark \ref{linear algebra}, an asymptotic weighted version of \eqref{davenportineq} is
$$1 \leq \liminf_{r \rightarrow \infty} \frac{\D_{2}^{h}(E_{p^{hr}})}{r} \leq 2.$$

Similar improvements for other values of $j$ have been presented in \cite{multiwise}.
\end{remark}

Note that since Theorem \ref{agthm} provides explicit constructions of short intersecting codes, combined with Lemma \ref{equivalence} it can be used to construct long sequences of elements of $E_{p^{hr}}$ with no $2$ disjoint weighted zero-sum subsequences.
This is stated precisely in the following remark.

\begin{remark} \label{davenportexplicitconstructions}
Let $p$ be a prime and let $h$ and $r$ be positive integers. There exist \emph{explicit} sequences of elements of $E_{p^{hr}}$ with no $2$ disjoint weighted zero-sum subsequences of length $\ell r$ with
\begin{itemize}
\item $\ell = 2 - \frac{2}{p^{m}}$, if $p \geq 5$, $h=2m$ and $p^h\geq 25$;
\item $\ell = 2 - \frac{2}{2 \frac{(p^{m} - 1) (p^{m+1}-1)}{(p^{m+1} + p^{m} -2)} + 1}$, if $h=2m+1$ and $p^h\geq 32$;
\item $\ell = \frac{3p-3}{2p-1}$, if $p\geq 11$ and $h=1$. 
\end{itemize}
The remaining cases are summarized in Table \ref{table5}.

\begin{table}[ht!]
\centering
\caption{Values of $\ell$ in the exceptional cases}
\label{table5}
\begin{tabular}{c|c|c|c}
$p$  & $h$ & $\ell$ & Probabilistic bound \\ \hline
$2$  & $1$ & $1.206$      & $1.261$            \\
$3$  & $1$ & $1.297$      & $1.365$            \\
$2$  & $2$ & $1.315$      & $1.424$            \\
$5$  & $1$ & $1.344$      & $1.464$            \\
$7$  & $1$ & $1.388$      & $1.517$            \\
$2$  & $3$ & $1.4$        & $1.535$            \\
$3$  & $2$ & $1.411$      & $1.551$            \\
$2$  & $4$ & $1.451$      & $1.614$            \\
$3$  & $3$ & $1.471$      & $1.660$           
\end{tabular}
\end{table}

\end{remark}

\bigskip

\section{Applications to factorization theory}\label{sec:factorization}

In the following section we will illustrate the impact of the previous results on factorization in the ring of integers of number fields. It is well-known that problems of factorization in ring of integers of number fields and more generally in Dedekind domains and Krull monoids are related to problems of zero-sum sequences in their class group (see \cite{Geroldinger2006NonUniqueF}). We will use some notions of algebraic number theory and we highlight their connection with the previous part of the paper. For the sake of brevity, we will not recall all the definitions, but we refer the interested reader to \cite{neukirch2013algebraic} for more details.

\medskip

\subsection{The classic scenario} Let $K$ be a number field, and let $\mathcal{O}_{K}$ be its integer ring. It is well-known that $\mathcal{O}_{K}$ is a Dedekind domain. We can define its ideal class group
$$\Cl(\mathcal{O}_{K}) = {\rm Frac}(\mathcal{O}_{K}) / {\rm Prin}(\mathcal{O}_{K})$$
where ${\rm Frac}(\mathcal{O}_{K})$ is the set of fractional ideals of $\mathcal{O}_{K}$ and ${\rm Prin}(\mathcal{O}_{K})$ its set of principal ideals.

For number fields it is well-known that the class group is finite and each class contains a prime ideal. In fact, all of the following assertions remain true for Dedekind domains (or even more generally for Krull monoids) with finite class group where each class contains a prime ideal.

Indeed, the original motivation for studying the Davenport constant stems from its connection to factorizations; see for example \cite{OLSON19698}, which already mentions the link recalled below between the Davenport constant of $\Cl(\mathcal{O}_{K})$ and factorizations in $\mathcal{O}_{K}$. 

\begin{lemma}\label{factorization}
The Davenport constant $\D(\Cl(\mathcal{O}_{K}))$ is the largest number of prime ideals (with multiplicities) occurring in the factorization of the ideal generated by an irreducible element $x \in \mathcal{O}_{K}$. Equivalently, the small Davenport constant $\dd(\Cl(\mathcal{O}_{K}))$ is the largest number of prime ideals such that their product is not divisible by a non-trivial principal ideal.
\end{lemma}

Since we will expand on this lemma, we consider it useful to provide the reader with a proof.

\begin{proof}
Let $(x) = \mathfrak{p}_{1} \cdot \ldots \cdot \mathfrak{p}_{n}$ be the unique factorization of $(x)$ as a product of prime ideals.
The image of this factorization in $\Cl(\mathcal{O}_{K})$ (considered with additive notation) is an identity of the form
$$0_{\Cl(\mathcal{O}_{K})} = [\mathfrak{p}_{1}] +\cdots + [\mathfrak{p}_{n}]$$
because $(x)$ is a principal ideal.
Note that the above identity means that $[\mathfrak{p}_{1}], \dots, [\mathfrak{p}_{n}]$ is a zero-sum sequence in $\Cl(\mathcal{O}_{K})$.
If this zero-sum subsequence can be decomposed into $2$ disjoint zero-sum subsequences, then $(x)$ is the product of $2$ (non-trivial) principal ideals, meaning that $x$ cannot be irreducible.
Therefore the zero-sum sequence must have length smaller than $\D(\Cl(\mathcal{O}_{K}))$.

Conversely, consider a zero-sum sequence in $\Cl(\mathcal{O}_{K})$ of length $n = \D(\Cl(\mathcal{O}_{K}))$ with no $2$ disjoint zero-sum subsequences (such a sequence must exist from the definition of the Davenport constant).
Such a sequence is of the form
$$0_{\Cl(\mathcal{O}_{K})} = [\mathfrak{a}_{1}] + \dots + [\mathfrak{a}_{n}].$$
Every class in $ \Cl(\mathcal{O}_{K})$ is represented by a prime ideal in $\mathcal{O}_{K}$. Every $[\mathfrak{a}_{n}]$ can therefore be represented by a prime ideal, say $\mathfrak{p}_{i}$.
Let $x \in \mathcal{O}_{K}$ be a generator of the principal ideal $\prod_{i = 1}^{n} \mathfrak{p}_{i}$.
Since there are no $2$ disjoint zero-sum subsequences, $x$ must be irreducible, and its unique factorization into prime ideals has length $\D(\Cl(\mathcal{O}_{K}))$.

Therefore $\D(\Cl(\mathcal{O}_{K}))$ is the largest number of prime ideals contained in the factorization of an irreducible element, as claimed.
\end{proof}

For every ideal $\mathcal{I}$ in $\mathcal{O}_{K}$, the ideal $\mathcal{I}^{\exp(\Cl(\mathcal{O}_{K}))}$ is principal.
In particular $(\mathfrak{p}_{1} \cdot \ldots \cdot \mathfrak{p}_{n})^{\exp(\Cl(\mathcal{O}_{K}))}$ is always principal.

In view of the preceding lemma, a natural question is: considering an ideal $\mathcal{I}$ in $\mathcal{O}_{K}$, which powers of $\mathcal{I}$ are divisible by a non-trivial principal ideal? Consider for example the case
$$\mathcal{I} = \prod_{i=1}^{n} \mathfrak{p}_{i}$$
where the $\mathfrak{p}_{i}$ are prime ideals. If $\mathcal{I}^{k}$ is divisible by a non-trivial principal ideal, then there must be a product
$$\prod_{i = 1}^{n} \mathfrak{p}_{i}^{\alpha_{i}}$$
(with $0 \leq \alpha_{i} \leq k$) which is a non-trivial principal ideal.
This can of course be interpreted as a $\{1, \dots, k\}$-weighted zero-sum subsequence in the ideal class group $\Cl(\mathcal{O}_{K})$.
The fully-weighted Davenport constant is a particular case of this general question, namely for $k = \exp(\Cl(\mathcal{O}_{K})) - 1$.

\begin{lemma}\label{weightedfactorization2}
The following interpretation of the different Davenport constants holds:
\begin{itemize}
\item for every $k \in \N$, the weighted small Davenport constant $\dd^{\{1, \dots, k\}}(\Cl(\mathcal{O}_{K}))$ is the largest number $\ell \in \N$ such that there exist $\ell$ prime ideals $\mathfrak{p}_{1}, \ldots, \mathfrak{p}_{\ell}$ such that the product
$$(\mathfrak{p}_{1} \cdot \ldots \cdot \mathfrak{p}_{\ell})^{k}$$
is not divisible by a non-trivial principal ideal;
\item the fully-weighted small Davenport constant $\dd^{f}(\Cl(\mathcal{O}_{K}))$ is the largest number $\ell \in \N$ such that there exist $\ell$ prime ideals $\mathfrak{p}_{1}, \ldots, \mathfrak{p}_{\ell}$ such that the product
$$(\mathfrak{p}_{1} \cdot \ldots \cdot \mathfrak{p}_{\ell})^{\exp(\Cl(\mathcal{O}_{K})) - 1}$$
is not divisible by a non-trivial principal ideal;
\item the $2$-wise small Davenport constant $\dd_{2}(\Cl(\mathcal{O}_{K}))$ is the largest number $\ell \in \N$ such that there exist $\ell$ prime ideals $\mathfrak{p}_{1}, \ldots, \mathfrak{p}_{\ell}$ such that the product
$$\mathfrak{p}_{1} \cdot \ldots \cdot \mathfrak{p}_{\ell}$$
is not divisible by a product of two non-trivial principal ideals (or equivalently, this product is divisible only by principal ideals generated by an irreducible element);
\item the $2$-wise fully-weighted small Davenport constant $\dd_{2}^{f}(\Cl(\mathcal{O}_{K}))$ is the largest number $\ell \in \N$ such that there exist $\ell$ prime ideals $\mathfrak{p}_{1}, \ldots, \mathfrak{p}_{\ell}$ such that the product
$$(\mathfrak{p}_{1} \cdot \ldots \cdot \mathfrak{p}_{\ell})^{\exp(\Cl(\mathcal{O}_{K})) - 1}$$
is not divisible by a product of two non-trivial principal ideals (or equivalently, this product is divisible only by principal ideals generated by an irreducible element).
\end{itemize}
\end{lemma}

\begin{proof}
The proof is a straightforward adaptation of the arguments of the proof of Lemma \ref{factorization}, taking into account the different definitions of the Davenport constants.
\end{proof}

We merely mention the $2$-wise case explicitly as the preceding section focused on this case. 

\subsection{The elementary abelian case: multiplicative action}
We are now going to focus only on the case when $ \Cl(\mathcal{O}_{K})$ is an elementary abelian group $E_{p^{hr}}$.
As in the previous section, writing $q = p^{h}$, we can consider a multiplicative action of $\Fq$ on $E_{p^{hr}}$.
Note that there is no unique way of defining such a multiplicative action (because there is no canonical isomorphism $\Fq^{r} \cong E_{p^{hr}}$).
However, our results do not depend on the choice of this isomorphism.

In this case, we have the following.

\begin{theorem}\label{thm:factorization}
Let $p$ be a prime and $h,r$ be positive integers. Let $K$ be an algebraic number field such that $\Cl(\mathcal{O}_{K}) = E_{p^{hr}}$.

The $2$-wise weighted small Davenport constant $\dd_{2}^{h}(\Cl(\mathcal{O}_{K}))$ is the largest number $\ell \in \N$ such that there exist $\ell$ prime ideals $\mathfrak{p}_{1}, \ldots, \mathfrak{p}_{\ell}$ such that any product
$$\prod_{i=1}^{\ell} \mathfrak{q}_{i},$$
where $\mathfrak{q}_{i}$ is an ideal in the class of $\varphi_{i}([\mathfrak{p}_i])$, with $\varphi_{i} \in \mathcal{Q}_{h}$, is not divisible by a product of two non-trivial principal ideals.
\end{theorem}

\begin{proof}
Again, the proof is an adaptation of the Lemma \ref{factorization}'s proof to the definition of $2$-wise weighted Davenport constant.
\end{proof}

\begin{remark}
It is quite remarkable that the above property does not depend on the chosen multiplicative action. It would be nice to have a general number-theoretical interpretation of such an invariant. At the end of this section, we will illustrate a link to the Galois action, which holds in some particular cases.
\end{remark}

Note that it is not known if every abelian group is the class group of the ring of integers of a number field. However, it is well-established that all (finite) abelian groups are the ideal class group of some Dedekind ring, as proved in \cite{Claborn1966EveryAG}. 
For illustrative purposes of the results above, we use the explicit construction presented in \cite[Theorem~2]{gerth1975number}, as well as explicit calculation in {\sc Magma}, to provide the following examples.

\begin{example}\label{exquad}
Let 
$$\alpha = 5\cdot 13\cdot 41\cdot 29\cdot 61$$
and $K = \Q(\sqrt{\alpha})$. We have that $\Cl(\mathcal{O}_{K})\cong E_{16}$. By Example \ref{exa:16} we know that $\dd_2(E_{16})=\D_2(E_{16})-1=7$, and $\dd_2^2(E_{16})=\D_2^2(E_{16})-1 = 5$.
Hence there exist $7$ prime ideals $\mathfrak{p}_1,\ldots,\mathfrak{p}_7$ such that their product is not divisible by a product of two non-trivial principal ideals and $5$ prime ideals $\mathfrak{p}_1,\ldots,\mathfrak{p}_5$ such that any product
$$\prod_{i=1}^{5} \mathfrak{q}_{i},$$
where $\mathfrak{q}_{i}$ is an ideal in the class of $\varphi_{i}([\mathfrak{p}_i])$, with $\varphi_{i} \in \mathcal{Q}_{2}$, is not divisible by a product of two non-trivial principal ideals.
\end{example}

\begin{example}
Let 
$$\alpha = 316861\cdot451897\cdot455333\cdot476977\cdot490549\cdot523793\cdot560641\cdot724481\cdot736993\cdot828829\cdot916621$$
and $K = \Q(\sqrt{\alpha})$. We have that $\Cl(\mathcal{O}_{K})\cong E_{1024}$. By Example \ref{exa:1024} we know that $\dd_2(E_{1024})=\D_2(E_{1024})-1=16$ and $\dd^2_2(E_{1024})=D^2_2(E_{1024})-1 = 10$.

Hence there exist $16$ prime ideals $\mathfrak{p}_1,\ldots,\mathfrak{p}_{16}$ such that their product is not divisible by a product of two non-trivial principal ideals and $10$ prime ideals $\mathfrak{p}_1,\ldots,\mathfrak{p}_{10}$ such that any product
$$\prod_{i=1}^{10} \mathfrak{q}_{i},$$
where $\mathfrak{q}_{i}$ is an ideal in the class of $\varphi_{i}([\mathfrak{p}_i])$, with $\varphi_{i} \in \mathcal{Q}_{2}$, is not divisible by a product of two non-trivial principal ideals.
\end{example}

\medskip

\subsection{The elementary abelian case: Galois group action}
By recent work \cite[Theorem 7.1]{boukheche2022monoids} it is known that the monoids of norms of rings of algebraic integers of Galois number fields admit a transfer homomorphism to monoids of weighted zero-sum sequences where the weights correspond to the elements of the Galois group. For a generalization, see \cite{geroldinger2022monoids}. Therefore studying weighted zero-sum problems over class groups with weights corresponding to the action of the Galois group has an immediate motivation from the point of view of factorization theory.

\medskip

It is well-known that the Galois group defines an action on the ideal class group (see Hilbert's Ramification Theory \cite[Chapter 1,\S 9]{neukirch2013algebraic}). It is then interesting to determine when this action is the same as that of $\mathcal{Q}_{h}$ defined above. This will give a natural interpretation of our definition of $\mathcal{Q}_{h}$ and our notion of generalized weights. Recall that the action of $\mathcal{Q}_{h}$ is the same as the multiplicative action of $\Fq$ on $E_{p^{hr}}$. The multiplicative group of $\Fq$ is the cyclic group $C_{q-1}$, and the orbits of the scalar action of $\Fq$ on $\Fq^{r}$ all have size $q-1$, that is the action is free on $\Fq^{r}\setminus\{0\}$. The following theorem shows that this is actually also a sufficient condition.

\begin{theorem}\label{thm:galois}
Let $p$ be a prime, and let $h$ and $r$ be positive integers, and let $q=p^h$. Let $K$ be a number field, with Galois group $\text{Gal}(K/\Q) = C_{q-1}$ and ideal class group $\Cl(\mathcal{O}_{K}) = E_{p^{hr}}$. If the action of ${\rm Gal}(K/\Q)$ is free on ${\rm Cl}(\mathcal{O}_K)\setminus\{0\}$, then there exists an isomorphism $\varphi: E_{p^{hr}} \rightarrow \Fq^{r}$ such that the action of $\text{Gal}(K/\Q)$ on the class group is the same as that of $\mathcal{Q}_{h}$.
\end{theorem}

\begin{proof}
First note that the action of the Galois group on the ideal class group preserves the group addition of $\Cl(\mathcal{O}_{K})$.

Let $\sigma$ be a generator of the Galois group. Observe that $E_{p^{hr}} \cong \mathbb{F}_{p}^{hr}$ as an $\F_p$-vector space.  It is clear that $\sigma$ corresponds to an endomorphism of $\mathbb{F}_{p}^{hr}$, which we also write $\sigma$. Since $x^q-x=0$ annihilates $\sigma$ and $\sigma$ has order $q-1$, the ring of endomorphisms $\mathbb{F}_{p}[\sigma]$ is isomorphic to $\F_q$.

Since the orbit of every non-zero element $v$ has order $q-1$, 
the map $$f(\sigma)=f_0+f_1\sigma+\ldots+f_{h-1}\sigma^{h-1}\mapsto f(\sigma)(v)=f_0v+f_1\sigma(v)+\ldots+f_{h-1}\sigma^{h-1}(v)$$ is a bijection between $\mathbb{F}_{p}[\sigma]$ and $\{0\}\cup \omega(v)$, which is also an isomorphism of $\F_p$-vector spaces. Via this map, we may endow $\{0\}\cup \omega(v)$ with a multiplicative structure that makes it isomorphic to $\F_q$.

Now, let $v_1$ be a non-zero element in $\Cl(\mathcal{O}_{K})$. Consider a non zero element $v_{2} \in \Cl(\mathcal{O}_{K}) \setminus  \omega(v_{1})$. As we said, both sets $\{0\} \cup \omega(v_{1})$ and $\{0\} \cup \omega(v_{2})$ are isomorphic to $\F_q$. The orbits are disjoint. Hence $W=\langle \omega(v_{1}), \omega(v_{2})\rangle \cong \Fq^{2}$. Moreover, $W$ is stable under the action of the Galois group. Now, we can continue by taking a nonzero element outside $W$ and so on, until getting 
$r$ elements, say $v_1,\ldots,v_r$. In this way, we have $\langle \omega(v_{1}), \ldots, \omega(v_{r})\rangle \cong \Fq^{r}$.
\end{proof}

\begin{remark}\label{rmk:mersenne}
If $p=2$ and $q-1=2^h-1$ is a Mersenne prime, then the action of ${\rm Gal}(K/\Q)$ is free on ${\rm Cl}(\mathcal{O}_K)\setminus\{0\}$. Actually, it is well-known that any prime $\ell\in \Z$ yields the following decomposition:
$$(\ell)=\left(\mathfrak{p}_1\cdots \mathfrak{p}_r\right)^e$$
and $er$ divides $q-1$. Hence, either $e=1$ and $r=q-1$, so that $\ell$ splits completely, or $e=q-1$ and $r=1$, so that $\ell$ is totally ramified, or $e=r=1$, so that $\ell$ is inert. Note in particular that when $(\ell)=(\mathfrak{p})^{q-1}$ (that is when $\ell$ is totally ramified), the ideal $\mathfrak{p}$ is principal. 
\end{remark}

To the best of our knowledge, it is unknown in general for which values of $p, h, r$ a number field satisfying all the hypotheses of Theorem \ref{thm:galois} exists. However, the following is an example, in the easiest case discussed in Remark \ref{rmk:mersenne}.

\begin{example}
Let $p(x) = x^{3} - x^{2} - 2562x + 48969$ and let $K=\Q[\alpha]$ be the cubic number field obtained extending $\Q$ with a root of $p(x)$. We have that 
$${\rm Gal}(K/\Q) = C_{3} \text{ and }\Cl(\mathcal{O}_{K})= E_{16}.$$
Moreover, $q-1=3$ is a Mersenne prime.
In this case, as in the Example \ref{exquad}, there exist $7$ prime ideals $\mathfrak{p}_1,\ldots,\mathfrak{p}_7$ such that their product is not divisible by a product of two non-trivial principal ideals and $5$ prime ideals $\mathfrak{p}_1,\ldots,\mathfrak{p}_5$ such that any product
$$\prod_{i=1}^{5} \mathfrak{q}_{i},$$
where $\mathfrak{q}_{i}$ is an ideal in the class of $\sigma([\mathfrak{p}_i])$, with $\sigma \in {\rm Gal}(K/\Q)$, is not divisible by a product of two non-trivial principal ideals. 
\end{example}

We conclude the paper with the following remark, which opens new perspectives for future research.

\begin{remark}    
There are number fields satisfying the hypotheses of Theorem \ref{thm:galois} which are not of prime degree, as in Remark \ref{rmk:mersenne}. For example, let $K=\Q[\alpha]$ where
$$\alpha^6 - \alpha^5 + 22\alpha^4 + 11\alpha^3 + 1038\alpha^2 - 1993\alpha + 16649=0.$$ In this case, 
$${\rm Gal}(K/\Q) = C_{6} \text{ and }\Cl(\mathcal{O}_{K})= E_{49}.$$
The action of ${\rm Gal}(K/\Q)$ is free on ${\rm Cl}(\mathcal{O}_K)\setminus\{0\}$ and the $8$ orbits of order $6$ are those of the following $8$ classes: $[\mathfrak{p}_\ell]$ where $\ell\in \{47,59,107,127,131,151, 173, 193\}$ and $\mathfrak{p}_\ell$ is a factor of $(\ell)$. 

It would certainly be interesting to further investigate fields that satisfy the hypotheses of Theorem \ref{thm:galois}, as well as to develop the coding-theoretical implications of a non-free action. This is certainly beyond the scope of the present paper, and it may be an interesting topic for  future researches.
\end{remark}

\bigskip

\noindent \textbf{Acknowledgements.} The three authors are partially supported by the ANR-21-CE39-0009 - BARRACUDA (French \emph{Agence Nationale de la Recherche}).
They would like to thank  Daniele Bartoli, Julien Lavauzelle and Alessandro Neri for the fruitful discussions on the topic and their insightful advice. The first and the third author would also like to warmly thank Inria GRACE team for hospitality and excellent working conditions, while this paper has mainly been written.

\medskip

\bibliography{articles}
\bibliographystyle{abbrv}

\end{document}